\documentclass[a4paper,11pt]{article}
\usepackage{amssymb}
\usepackage{amsfonts}
\usepackage{amsmath}
\usepackage{amsthm}
\usepackage{graphicx}

\def\es{\emptyset}
\def\eps{\varepsilon}

\def\dist{\mathrm{dist}}
\def\maxim{\mathrm{max}}

\def\ort{\mathrm{ort}}

\def\Var{\mathsf{Var}}

\def\cI{\mathcal{I}}
\def\bone{\mathbf{1}}
\def\bzero{\mathbf{0}}

\def\bn{\mathbf{n}}
\def\bN{\mathbf{N}}
\def\bp{\mathbf{p}}
\def\br{\mathbf{r}}
\def\bs{\mathbf{s}}
\def\cB{\mathcal{B}}
\def\cF{\mathcal{F}}
\def\cG{\mathcal{G}}
\def\cK{\mathcal{K}}
\def\cR{\mathcal{R}}
\def\cS{\mathcal{S}}

\def\bu{\mathbf{u}}
\def\bw{\mathbf{w}}
\def\bx{\mathbf{x}}
\def\bX{\mathbf{X}}
\def\by{\mathbf{y}}
\def\bY{\mathbf{Y}}
\def\bz{\mathbf{z}}

\def\P{\mathbf{P}}
\def\E{\mathbf{E}}
\def\N{\mathbb{N}}
\def\R{\mathbb{R}}

\newcommand{\eqn}[2]{\begin{equation}\label{#1}#2\end{equation}}
\newcommand{\eqnst}[1]{\begin{equation*}#1\end{equation*}}
\newcommand{\eqnspl}[2]{\begin{equation}\begin{split}\label{#1}%
    #2\end{split}\end{equation}}
\newcommand{\eqnsplst}[1]{\begin{equation*}\begin{split}%
    #1\end{split}\end{equation*}}

\theoremstyle{plain}
\newtheorem{theorem}{Theorem}
\newtheorem{lemma}[theorem]{Lemma}
\newtheorem{proposition}[theorem]{Proposition}
\newtheorem{corollary}[theorem]{Corollary}

\theoremstyle{definition}

\newtheorem{problem}{Question}

\theoremstyle{remark}
\newtheorem*{remark}{Remark}
\newtheorem*{conjecture}{Conjecture}

\begin{document}

\title{Phase transition in a sequential assignment problem on graphs}
\author{Antal A.\ J\'arai\thanks{Department of Mathematical 
Sciences, University of Bath,
Claverton Down, Bath BA1 7AY, United Kingdom. 
E-mail: {\tt A.Jarai@bath.ac.uk}}}

\maketitle

\begin{abstract}
We study the following sequential assignment problem on a finite 
graph $G = (V, E)$. Each edge $e \in E$ starts with an integer value
$n_e \ge 0$, and we write $n = \sum_{e \in E} n_e$. At time $t$,
$1 \le t \le n$, a uniformly random vertex $v \in V$ is generated,
and one of the edges $f$ incident with $v$ must be selected. 
The value of $f$ is then decreased by $1$. There is a unit final 
reward if the configuration $(0, \dots, 0)$ is reached. Our main result 
is that there is a \emph{phase transition}: as $n \to \infty$,
the expected reward under the optimal policy approaches a constant 
$c_G > 0$ when $(n_e/n : e \in E)$ converges to a point in the 
interior of a certain convex set $\cR_G$, and goes to $0$ exponentially 
when $(n_e/n : e \in E)$ is bounded away from $\cR_G$. We also obtain
estimates in the near-critical region, that is when 
$(n_e/n : e \in E)$ lies close to $\partial \cR_G$. 
We supply quantitative error bounds in our arguments.
\end{abstract}


\section{Introduction}
\label{sec:intro}

Consider the following game (known in different versions 
\cite{J16}, \cite[Section 1.7]{Pbook}). 
Players start with a row of $N$ empty boxes. In each of $N$ rounds, 
a random digit is generated, and each player has to place it 
into one of the empty boxes they have. A player's score is 
the $N$ digit number obtained after the last round. 
The game is a special case of \emph{sequential stochastic assignment}
introduced by Derman, Lieberman and Ross \cite{DLR72}. In sequential 
assignment, there are $N$ jobs with given values $p_1 \le \dots \le p_N$
that have to be assigned to $N$ workers, as they appear in sequence. 
The $i$-th worker has ability $X_i$, where $X_1, \dots, X_N$ are i.i.d.~random 
variables from a given distribution $F$. The reward from assigning 
the job of value $p_i$ to a worker with ability $x$ is $p_i x$,
and the overall reward of the assignment is the sum of the individual
rewards. The game mentioned at the start is recovered when 
$p_i = 10^{i-1}$, and $X_i$ is uniform in $\{ 0, \dots, 9 \}$. 

The paper \cite{DLR72} showed that there is a strategy that maximizes the 
expected score independently of what $p_1, \dots, p_N$ are. This 
strategy has the following form. There are numbers 
$-\infty = a_{0,n} \le a_{1,n} \le \dots \le a_{n-1,n} \le a_{n,n} = \infty$,
$n \ge 1$, that only depend on the distribution $F$, such that if there are
$n$ jobs remaining to be assigned, with values $p'_1 \le \dots \le p'_n$,
and the next worker has ability $x$ with 
$a_{i-1,n} \le x \le a_{i,n}$, then the worker is assigned to the
job with value $p'_i$.

Albright and Derman \cite{AD72} showed, using law of large numbers type arguments, that when $F$ is absolutely 
continuous, one has $\lim_{n \to \infty} a_{q n, n} = F^{-1}(q)$, $0 < q < 1$,
as $n \to \infty$. In particular, when the number $n$ of jobs is large,
a worker with ability $x$ should be assigned to a job with rank approximately 
$q n$, where $F^{-1}(q) = x$. Note that when $F$ is discrete, 
this way of determining the asymptotics breaks down: when $x$
is an atom of $F$, the graph of $F^{-1}$ has a horizontal
piece at height $x$. For large finite $n$, the value of $q$
where the profile $a_{q n, n}$ crosses height $x$ can be expected to be somewhere in the corresponding interval of constancy of $F^{-1}$, and its precise location can be expected to be governed by large deviation effects.

In order to motivate the subject of our paper, consider the 
following modification of the game mentioned at the beginning. Suppose that each digit can take 
the values $1, \dots, k$, with equal probability. Also suppose that 
the goal of the player is to maximize the probability of achieving 
the maximum possible score, that is to reach the unique final assignment 
consisting of $k$ contiguous intervals of equal digits. 
Let $\tau$ be the first time when all $k$ numbers have occurred
at least once. At time $\tau$, the empty boxes form $k-1$ intervals
of lengths $n_1, \dots, n_{k-1}$, where $n - \tau = \sum_{i=1}^{k-1} n_i$. 
The $i$-th interval has a box filled with $i$ adjacent to it on the right, and
a box filled with $i+1$ adjacent to it on the left.
It is plausible that there exist numbers 
$0 = \alpha_1 < \alpha_2 < \dots < \alpha_{k-1} < \alpha_{k} = 1$, such that for large $n$, under the optimal strategy, 
$n_i/n \sim \alpha_{i+1} - \alpha_i$, $i = 1, \dots, k-1$.
We will be interested in the following question.
Suppose that an 
alternative position is imposed on the player, where the intervals have
length $n'_i \sim (\beta_{i+1} - \beta_i) n'$, $i = 1, \dots, k-1$, where
$0 = \beta_1 < \beta_2 < \dots < \beta_{k-1} < \beta_k = 1$. 
What is the behaviour of the probability that the player
can achieve the maximal score \emph{from this position}?

We show that the above probability displays a 
sharp transition in the limit $n' \to \infty$. 
When the vector $(\beta_{i+1} - \beta_i : i = 1, \dots, k-1)$ 
lies in the interior of a certain convex set $\cR_k$, the probability approaches a positive constant, 
whereas it goes to $0$ exponentially when 
the vector is at a positive distance from $\cR_k$.

More generally, we consider the above transition on a general finite graph $G = (V,E)$
with vertices labelled $1, \dots, k$. The starting position is a vector 
$(n_e : e \in E)$, and $n = \sum_{e \in E} n_e$. When a number 
$1 \le i \le k$ is rolled, one of the edges $f$ incident with vertex $i$
is selected by the player, and the value assigned to edge $f$ is decreased
by $1$. We assign a final reward of $1$ when the configuration 
$(0, \dots, 0)$ is reached, and refer to this as `winning'.
In the game described at the beginning, the graph is a path of length $k-1$.

We believe the study of this model is interesting for a number of reasons.
\begin{enumerate}
\item Questions of reachability have been studied in control theory for a 
long time \cite[Sections 19,20]{PBGMbook}. In our model, the controllable
set $\cR_G$, that allows the player to reach the state $(0,\dots,0)$ with
uniformly positive probability, has a simple characterization, 
which however involves the graph structure in a non-trivial way; 
see Eqn.~\eqref{e:RG-def} and Lemma \ref{lem:prop-R_G}. As we show, 
choosing the right control is only essential near $\partial \cR_G$.
We believe our model, that is tractable on a general graph, is a useful
example system to have in understanding the behaviour of discrete controlled 
systems with spatial structure near critical regions. Indeed, the main technical
effort in this paper is getting estimates in the near critical region,
that we do in Section \ref{sec:critical-proof}.
\item In deriving the optimal strategy for sequential assignment, 
Derman, Lieberman and Ross \cite{DLR72} used 
Hardy's inequality, of which we have no analogue on graphs.
Our proofs work without knowledge of the optimal strategy,
and only rely on martingale and Lyapunov function techniques, as well
as an explicit relationship between $\cR_G$ and available controls.
Thus our arguments may be adaptable to other models.
It may be that the transition phenomenon itself can be
established with less effort, given more information 
on the optimal strategy (see for example Question \ref{prob:exp-converge} 
in Section \ref{sec:open}). Nevertheless, we believe that the
quantitative bounds we derive are of independent interest.
\item As the title of this paper suggests, we view the
transition studied in this paper as an instance of a critical
phenomenon.\footnote{A reader unfamiliar with critical phenomena 
can find a good introduction in the short text \cite{Gbook10}. 
We note that such familiarity is not required for understanding 
this paper.} 
While such transitions are ubiquitous in stochastic 
control, we found little in the literature that connects them 
with critical phenomena. We believe that such a point of view 
can be beneficial, and was indeed our original motivation for
this study. Examples of works in the physics 
literature that address an interplay between controllability 
and network structure are \cite{NV12,JLCPSB13,SM13}.
\item Further problems that are important for applications can be studied
in our model or suitable modifications thereof. For example, we
see no obvious \emph{distributed} control, where vertices would 
only have local information about the graph structure.
\end{enumerate}

\subsection{Definition of the model}
\label{ssec:model}

Throughout $G = (V,E)$ will be a finite connected simple graph
(without multiple edges or loops). We write 
$k = |V|$, and assume $|E| \ge 2$ (the case with 
one edge being trivial). 
We write $\deg_G(v)$ for the degree of $v \in V$, and 
$\deg_F(v)$ for the degree of $v$ in the 
subgraph of $G$ induced by the set of edges $F \subset E$. 

The state at time $0 \le t \le n$ is an integer vector 
$\bN(t) = (N_e(t) : e \in E)$, where the starting
state is $\bN(0) = \bn = (n_e : e \in E)$. Usually we will 
use capitalized letters for random variables or random 
processes, and lowercase letters for their possible values.
We write
$n = \sum_{e \in E} n_e$. Let $V_1, \ldots V_n \in V$ be an
i.i.d.~sequence of vertices with $\P [ V_i = v ] = \frac{1}{k}$, 
$v \in V$, $i = 1, \dots, n$.
If the player allocates $V_t$ to the edge $e$ incident
with $V_t$, the state is updated as
\eqnst
{ \bN(t) = \bN(t-1) - \bone^e, \quad \text{ where } \quad
  \bone^e = (1^e_f : f \in E), \quad 
  1^e_f 
  = \begin{cases}
    1 & \text{if $f = e$;}\\
    0 & \text{if $f \not= e$.} 
    \end{cases} }
The gambler wins if $\bN(n) = (0, \dots, 0) \in \N^E$, 
and looses otherwise. We denote by $p_G(\bn)$ 
the probability of winning under the optimal strategy, when the
starting state is $\bn$. This satisfies
\eqn{e:optimality}
{ p_G(\bn)
  = \frac{1}{k} \, \sum_{v \in V} \, \max_{e \in E : e \sim v} \, p_G(\bn - \bone^e), }
known as the \emph{optimality equation} \cite[Section I.1]{Rbook},
where $e \sim v$ means that $e$ is incident with $v$.

We introduce some notation needed to state our main theorem.
We write $\cS_G$ for the probability simplex in $\R^E$, that is, the set 
of non-negative vectors $\bx \in \R^E$ such that $\sum_{e \in E} x_e = 1$. 
We define
\eqnspl{e:RG-def}
{ d(F)
  &= \left| \{ v \in V : \deg_F(v) = \deg_G(v) \} \right|, \quad
     \es \subset F \subset E; \\
  \cR_G
  &= \left\{ \bx \in \cS_G : \text{for all 
     $\es \subsetneq F \subsetneq E$ we have $\sum_{e \in F} x_e > 
     \frac{1}{k} d(F)$} \right\}; \\
  \cI_G
  &= \left\{ \bx \in \cS_G: \text{there exists
     $\es \subsetneq F \subsetneq E$ such that $\sum_{e \in F} x_e <
      \frac{1}{k} d(F)$} \right\}. }
The letters `$d$', `$\cR$' and `$\cI$' are intended to evoke 
`degree', `reachable' and `inaccessible', as we explain.
For any non-empty set 
$F$ of edges, $\frac{d(F)}{k}$ is the probability that the player receives 
a vertex that has full degree in $F$. 
Any such vertex \emph{must} be allocated to one of the edges in $F$. 
For starting positions $\bn = (n_e : e \in E)$ where the proportion 
of space $\sum_{e \in F} n_e / n$ available at the beginning 
is smaller than $d(F)/k$, the probability 
of winning goes to $0$ (as $n \to \infty$).
Therefore, from the region $\cI_G$ the winning
position is asymptotically inaccessible.
On the other hand, as we show in Theorem \ref{thm:phase-trans-graph},
if $\bn = n \, \bx$ with $\bx \in \cR_G$, then the winning position
is asymptotically reachable from $\bn$.
As we point out in Section \ref{ssec:prelim}, the set $\cR_G$ 
arises as the region of controllability for a simple
(deterministic) linear control system associated to the game.
It can be verified that when $G$ is a tree with $k$ vertices 
($k \ge 3$) $\cR_G$ is a parallelepiped. 
As we will not need this fact, we omit the proof.

\begin{remark}
The arguments we present in this paper are also applicable 
to the slightly more general model when $V_1, \dots, V_n$ are not
uniformly distributed (but still i.i.d.). Suppose 
$\P [ V_i = v ] = p_v$ with a probability vector 
$\bp = (p_v : v \in V)$ such that $p_v > 0$ for all $v \in V$.
In this case $\cR_G$ and $\cI_G$ are replaced by 
\eqnsplst
{ \cR_{G,\bp}
  &= \left\{ \bx \in \cS_G : \text{for all 
     $\es \subsetneq F \subsetneq E$ we have $\sum_{e \in F} x_e > 
     \sum_{v : \deg_F(v) = \deg_G(v)} p_v$} \right\}; \\
  \cI_{G,\bp}
  &= \left\{ \bx \in \cS_G: \text{there exists
     $\es \subsetneq F \subsetneq E$ such that $\sum_{e \in F} x_e <
     \sum_{v : \deg_F(v) = \deg_G(v)} p_v$} \right\}, }
As the required changes in the proofs are minor, but including them would
burden the notation further, we state and prove the results only
in the uniform case. All the essential difficulties are already present
in the uniform model.
\end{remark}

\subsection{Main results}
\label{ssec:main-result}

Theorems \ref{thm:phase-trans-graph} and \ref{thm:bdry} below state
our main results. Figure \ref{fig:k4-plot} illustrates these 
when $G$ is a path of length three, that is $k = 4$. 

\begin{theorem}
\label{thm:phase-trans-graph}
Let $G$ be a finite connected simple graph with $|E| \ge 2$. \\
(i) If $\bx \in \cI_G$, and $\bn = n \bx + O(1)$, then
$p_G(\bn) \to 0$ exponentially fast, as $n \to \infty$, 
at a rate depending on $\bx$. The rate of decay is bounded
away from $0$ on subsets bounded away from $\cR_G$. \\
(ii) There exists a constant $c_G > 0$, such that 
if $\bx \in \cR_G$, and $\bn = n \bx + O(1)$, then
$p_G(\bn) \to c_G$, as $n \to \infty$.
\end{theorem}

\begin{figure}[htpb]
(a) \hskip-1cm\includegraphics[scale=0.7]{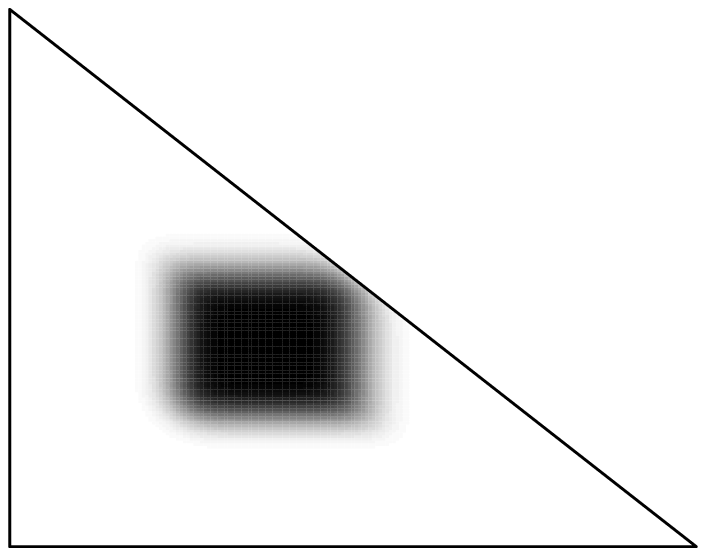} \
(b) \hskip-1cm\includegraphics[scale=0.7]{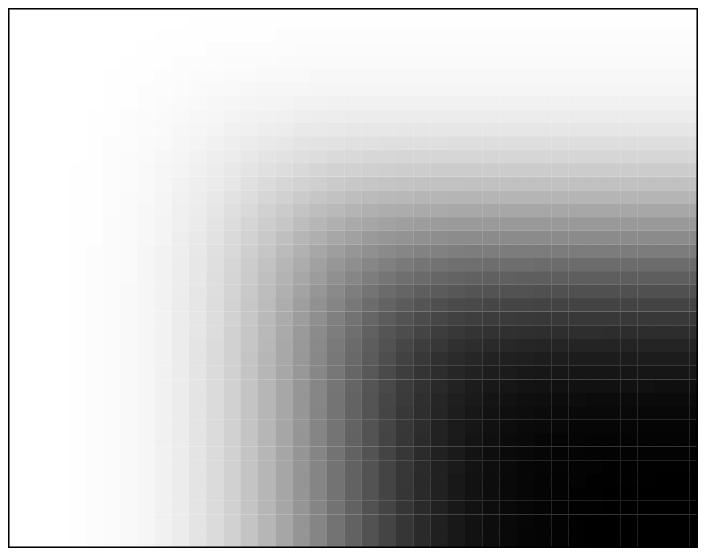}
\caption{%
(a) Image of $p_G(m, 200-m-\ell, \ell)$ when $G$ is a path of length 
three ($k = 4$) and $n = 200$. 
The limit of $p_G$ is a positive constant in the rectangle
$\frac{1}{4} < x = m/n, y=\ell/n < \frac{1}{2}$ (dark region), 
and goes to $0$ when $(x,y)$ is away from the rectangle (white region). 
The maximum of $p_G$ is $\approx 0.2583299$.
(b) Detailed image of $p_G$ near the corner of the critical region
$0.15 \le m/n \le 0.35$, $0.4 \le \ell/n \le 0.6$.}
\label{fig:k4-plot}
\end{figure}

In Section \ref{sec:critical-proof} we obtain bounds on the 
behaviour near $\partial \cR_G$. These shows that the 
`critical window' has width of order $\sqrt{n}$ around
$n \partial \cR_G$. Our bounds in particular imply the following
upper bound on $p_G(\bn)$ in this region. Fix any $\delta > 0$, and let 
\eqnst
{ \overline{M}_n
  = \overline{M}_n(\delta) 
  = \max \left\{ p_G(\bn) : \bn/n \in \cS_G,\, 
    \dist(\bn/n, \partial \cR_G) \le \delta \right\}. }

\begin{theorem}
\label{thm:bdry}
For any $\delta > 0$ we have
$\limsup_{n \to \infty} \overline{M}_n(\delta) \le c_G$.
\end{theorem}

Combining Theorems \ref{thm:phase-trans-graph} and \ref{thm:bdry}
we obtain the following corollary.

\begin{corollary}
\label{cor:max-pG}
The configuration $\bn$ that maximizes $p_G(\bn)$ with $n$ fixed,
satisfies $p_G(\bn) = c_G + o(1)$, as $n \to \infty$.
\end{corollary}

Theorems \ref{thm:phase-trans-graph} and \ref{thm:bdry} do not rule 
out the possibility that $p_G(\bn)$ is maximized near the critical surface,
at a distance that is $o(n)$. But of course we expect that 
the location of the maximum, when rescaled by $1/n$, converges to a point in the 
interior of $\cR_G$. It is also plausible that the location of this point
can be characterized in terms of large deviation rates for events 
of the form `the gambler runs out of space on the edges in $F$', 
that is:
\eqnst
{ \left\{ \sum_{v : \deg_F(v) = \deg_G(v)} \sum_{t=1}^n \mathbf{1}_{V_t = v}
     > \sum_{e \in F} n_e \right\}, \quad
     \es \subsetneq F \subsetneq E. }
We state an explicit conjecture for a path of length $k-1$, where this
is easiest to formulate.
Let 
\eqnst
{ a_*(j;k)
  = \frac{\log \left( \frac{k-j-1}{k-j} \right)}{\log \left( 
    \frac{j \, (k-j-1)}{(j+1) \, (k - j)} \right)}, \quad
    1 \le j \le k-2 \qquad\quad
    a_*(0;k) = 0 \qquad\quad
    a_*(k-1;k) = 1. }
Let $\bn^{\max} = (n^{\max}_j : j = 1, \dots, k-1)$ denote a point in $n \, \cS_G$ where
$p_G(\bn)$ is maximized, $n \ge 1$.

\begin{conjecture}
\label{conj:path}
Let $k \ge 3$. Then for $1 \le j \le k-2$ we have
\eqnst
{ \lim_{n \to \infty} \frac{1}{n} \sum_{\ell=1}^j n^{\max}_j
  = a_*(j;k). }
\end{conjecture}

The number $a_*(j;k)$ is obtained as the unique point 
$a \in \left( \frac{j}{k}, \frac{j+1}{k} \right)$, for which the 
`cheaper' of the two large deviation events
\eqnst
{ \left\{ \sum_{v=1}^j \sum_{t=1}^n \mathbf{1}_{V_t = v}
     >  a \, n \right\} \quad \text{ and } \quad
  \left\{ \sum_{v=j+2}^k \sum_{t=1}^n \mathbf{1}_{V_t = v}
     >  (1-a) \, n \right\} }
is as `expensive' as possible. (This number $a$ can be obtained by equating
the large deviation rates of the two events.) Each $a_*(j;k)$ marks
out a linear submanifold of $\cS_G$, and the location of the optimum
is their intersection. We expect that a similar characterization 
holds for any connected graph $G$.

The structure of the paper is as follows.
The proof of Theorem \ref{thm:phase-trans-graph} is given in 
Section \ref{sec:main-proof}. We study the behaviour near $\partial \cR_G$ 
in Section \ref{sec:critical-proof}, and deduce Theorem \ref{thm:bdry}.
We stress however, that our analysis provides a much more refined picture 
than Theorem \ref{thm:bdry}; see Propositions \ref{prop:far-enough},
\ref{prop:in-enough} and \ref{prop:near-critical}, and their proof. The estimates
in these propositions suggest Gaussian behaviour near $\partial \cR_G$.
We conclude with some further questions in Section \ref{sec:open}.

\section{Proof of the phase transition}
\label{sec:main-proof}

The next section collects some preliminaries and useful notation.

\subsection{Basic properties of $\cR_G$}
\label{ssec:prelim}

It will be convenient to have the version of $\cR_G$ in which the 
inequalities are not strict:
\eqnst
{ \cK_G
  = \left\{ \bx \in \cS_G : \text{for all $F \subset E$
     we have $\sum_{e \in F} x_e \ge \frac{1}{k} d(F)$} \right\}. }
We denote by $H_F$ the hyperplanes appearing in these inequalities:
\eqnst
{ H_F
  = \left\{ \bx \in \R^E : \sum_{e \in F} x_e = \frac{1}{k} d(F) \right\},
    \es \not= F \subset E. }
In particular, $\cS_G$, $\cR_G$, $\cI_G$ and $\cK_G$ are all subsets of $H_E$.

\begin{lemma}
\label{lem:prop-R_G} \ \\
(i) The sets $\cK_G$ and $\cR_G$ are convex with a non-empty interior
relative to $H_E$.\\
(ii) $\cK_G = \overline{\cR_G}$ (the closure of $\cR_G$ in $H_E$).
\end{lemma}

\begin{proof}
(i) As intersections of halfspaces with $H_E$, both $\cK_G$ and $\cR_G$ are convex. 
Also, since the halfspaces defining $\cR_G$ (resp.~$\cK_G$) are open (resp.~closed), 
$\cR_G$ (resp.~$\cK_G$) is a relatively open (resp.~closed) subset of $H_E$.
The containment $\cR_G \subset \cK_G$ is immediate from the definitions.
To show that $\cR_G$ has non-empty interior, we check that the vector
\eqn{e:x^*}
{ \bx^*
  = (x^*_e : e \in E), \quad
  x^*_e 
  = \frac{1}{k} \sum_{\substack{v \in V \\ v \sim e}} \frac{1}{\deg(v)}, 
    \quad e \in E, }
belongs to $\cR_G$. First, $\bx^* \in H_E$ can be seen by summing the
formula for $x^*_e$ over $e \in E$ and exchanging the two sums. It is also
immediate that $x^*_e > 0$, and therefore $\bx^* \in \cS_G$.
Now fix any $\es \subsetneq F \subsetneq E$. Since
$G$ is connected, there exists a vertex $v \in V$ such that 
$0 < \deg_F(v) < \deg_G(v)$. Therefore,
\eqnsplst
{ \sum_{e \in F} x^*_e 
  &= \sum_{e \in F} \frac{1}{k} \sum_{\substack{v \in V \\ v \sim e}} \frac{1}{\deg(v)}
  = \frac{1}{k} \sum_{\substack{v \in V \\ \deg_F(v) = \deg_G(v)}} 
    \sum_{\substack{e \in F \\ e \sim v}} \frac{1}{\deg_G(v)} 
    + \frac{1}{k} \sum_{\substack{v \in V \\ \deg_F(v) < \deg_G(v)}} 
    \sum_{\substack{e \in F \\ e \sim v}} \frac{1}{\deg_G(v)} \\
  &> \frac{1}{k} \sum_{\substack{v \in V \\ \deg_F(v) = \deg_G(v)}} 1 
  = \frac{d(F)}{k}. }
This shows that $\bx^* \in \cR_G$, and since $\cR_G$ is open in $H_E$,
$\bX^*$ is an interior point. The containment $\cR_G \subset \cK_G$ implies
that $\bx^*$ is also an interior point of $\cK_G$.
 
(ii) Since $\cK_G$ is closed, we have $\overline{\cR_G} \subset \cK_G$.
Therefore, it is enough to show that $\cK_G \setminus \cR_G \subset \overline{\cR_G}$.
Let $\bx \in \cK_G \setminus \cR_G$. Let $\bx(t) = t \bx + (1-t) \bx^*$.
Convexity of $\cK_G$ implies that $\bx(t) \in \cK_G$ for all $0 \le t \le 1$.
Moreover, since the expressions $\sum_{e \in F} x_e(t)$ are monotone linear functions
of $t$, and $\sum_{e \in F} x_e(0) > d(F)/k$, and $\sum_{e \in F} x_e(1) \ge d(F)/k$, 
we must have the inequality $\sum_{e \in F} x_e(t) > \frac{1}{k} d(F)$ 
for all $0 \le t < 1$. This implies that $\bx(t) \in \cR_G$ for 
$0 \le t < 1$, and hence $\bx \in \overline{\cR_G}$, as required.
\end{proof}

The optimality equation implies that the optimal deterministic strategy is 
also optimal among randomized strategies. The next lemma states a 
connection between elements of $\cK_G$ and possible moves in a
randomized strategy. In its statement, we think of $q^{(v)}(e)$ as the
probability of assigning vertex $v$ to the edge $e$ in such a move.

\begin{lemma}
\label{lem:interpret}
We have $\bx \in \cK_G$ if and only if there exists a
collection $\{ q^{(v)}(e) : v \in V,\, e \in E \}$
of non-negative numbers such that:\\
(i) $\sum_{e \in E} q^{(v)}(e) = 1$ for all $v \in V$; \\
(ii) $q^{(v)}(e) = 0$ if $e$ is not incident with $v$; \\
(iii) $\frac{1}{k} \sum_{v \in V} q^{(v)}(e) = x_e$ for all $e \in E$.
\end{lemma}

\begin{proof}
We deduce the statement from the Max-Flow-Min-Cut Theorem \cite[Theorem III.1]{Bbook}. 
Define an auxilliary
directed graph $G'$ as follows. Replace each edge $\{ v, w \}$ 
of $G$ by two directed edges $( v, u_e )$ and $( w, u_e )$, introducing 
the new vertex $u_e$ for each $e \in E$. Also add new vertices 
$s$ and $t$. Add a directed edge $( s, v )$ for each $v \in V$ and
a directed edge $( u_e, t )$ for each $e \in E$. Thus $G'$ has
$|V| + |E| + 2$ vertices and $2 |E| + |V| + |E|$ edges.

Consider flows of strength $1$ from $s$ to $t$ in $G'$, where
we assign capacity $1/k$ to each edge $(s, v)$, $v \in V$, 
capacity $2$ to each $(v, u_e)$ and capacity $x_e$ to 
each $(u_e, t)$.

Suppose $q^{(v)}(e)$ satisfy (i)--(iii). Define a flow by letting
$1/k$ flow on each $(s,v)$, $q^{(v)}(e)/k$ flow on each $(v, u_e)$, 
and $x_e$ flow on each $(u_e,t)$. This flow satisfies
the capacity constraints, and it is a maximal flow, since
$\{ (s, v) : v \in V \}$ is a cut with value $1$. Therefore any
other other cut must have value at least $1$. 
Given $\es \subset F \subset E$, consider the cut
\eqn{e:cut}
{ \{ (s, v) : \deg_F(v) < \deg_G(v) \}
     \cup \{ (u_e, t) : e \in F \}. }
with value
\eqnst
{ \frac{k - d(F)}{k} + \sum_{e \in F} x_e
  = 1 - \frac{d(F)}{k} + \sum_{e \in F} x_e 
  \ge 1. }
This implies that $\bx \in \cK_G$.

For the converse, suppose that $\bx \in \cK_G$, and 
consider a maximal flow on $G'$. The conditions in the 
definition of $\cK_G$ imply that all cuts of the form 
\eqref{e:cut} have value $\ge 1$, and the cut corresponding 
to $F = E$ has value $1$. It is easy to check 
that any minimal cut is necessarily of this form, and therefore
the maximal flow is $1$. Letting $q^{(v)}(e)$ be $k$-times the
amount flowing on $(v, u_e)$ we obtain a collection satisfying
(i)--(iii).
\end{proof}

Basic for Theorem \ref{thm:phase-trans-graph} is the following 
computation. Suppose that our current state is 
$\bn = n \bx$, $\bx \in \cS_G$. 
Let $\{ q^{(v)}(e) \}_{v \in V, e \in E}$ be a set of 
probabilities representing a randomized move (that is:
$q^{(v)}_e$ is the probability that edge $e$ will be used,
conditional on the event that vertex $v$ has been drawn). Let 
$\bN' = (n-1) \bX'$ be the random outcome of the move. 
Let $y_e = \frac{1}{k} \sum_{v \in V} q^{(v)}(e)$. We have
\eqnspl{e:drift1}
{ \E \bX' 
  &= \frac{1}{n-1} \E \bN'
  = \frac{1}{n-1} \left( \bn - \sum_{e \in E} y_e \bone^e \right) 
  = \frac{n}{n-1} \bx - \frac{1}{n-1} \by
  = \bx + \frac{1}{n-1} ( \bx - \by ). }
If $\bx \in \cR_G$, then due to Lemma \ref{lem:interpret} 
it is possible to choose $\by \in \cK_G$ in such a 
way that the average displacement 
points in any desired direction. 
On the other hand, if $\bx \in \cI_G$, convexity of $\cK_G$ implies that 
the process will always move away from $\cR_G$ on average.

The above observations are also reflected in the following deterministic
controlled differential equation:
\eqnst
{ \frac{d \bx}{dt}
  = \bx - \bu(t), \quad \text{where the control $\bu$ satisfies 
    $\bu(t) \in \cK_G$ for all $t \ge 0$.} }
It is easy to see (for example using as Lyapunov function the distance 
from $H_E \cap H_F$ for suitable $F$) that:\\
(i) If $\bx(0) \not\in \cK_G$, then for any control $\bu$ we have
$\bx(t) \not\in \cK_G$ for all $t \ge 0$;\\
(ii) If $\bx(0) \in \cR_G$, then for any $\bx' \in \cR_G$ there exists a 
control $\bu$ such that $\lim_{t \to \infty} \bx(t) = \bx'$.

Let us introduce some further notation.
Throughout we write $\| \bw \|_1 = \sum_{e \in E} |w_e|$ and
$| \bw | = \sqrt{\sum_{e \in E} |w_e|^2}$ for any 
vector $\bw = (w_e : e \in E) \in \R^E$. For $\bw \in \R^E$ and
$A \subset \R^E$ we write $\dist(\bw, A) = \inf_{\by \in A} |\bw - \by|$.
We will write $\langle \cdot, \cdot \rangle$ for the Euclidean
scalar product.

For each $\es \subsetneq F \subsetneq E$ we fix a point $\bz^F \in \cK_G$
such that $\sum_{e \in F} z^F_e = \frac{d(F)}{k}$. Let $\bu^F$ be the
unit vector of the form 
\eqnst
{ u^F_e
  = \begin{cases}
    a_F & \text{if $e \in F$;}\\
    -b_F & \text{if $e \in E \setminus F$,}
    \end{cases} } 
with $a_F, b_F > 0$, and such that 
$\sum_{e \in E} u^F_e = 0$. For all $\bw \in \cK_G$ we have
$\langle \bw - \bz^F, \bu^F \rangle \ge 0$. We will often use linear
functions of the form:
\eqnsplst
{ L^{F,n}(\bn)
  = \langle \bn - n \bz^F, \bu^F \rangle
  = \sum_{e \in E} (n_e - n z^F_e) u^F_e. }
The last expression can be rewritten as follows:
\eqnsplst
{ &\sum_{e \in E} (n_e - n z^F_e) u^F_e
  = a_F \sum_{e \in F} n_e + (-b_F) \left( n - \sum_{e \in F} n_e \right)
    - n a_F \sum_{e \in F} z^F_e - n (-b_F) \left( 1 - \sum_{e \in F} z^F_e \right) \\
  &\qquad = (a_F + b_F) \sum_{e \in F} n_e - n b_F - n (a_F + b_F) \sum_{e \in F} z^F_e + n b_F 
  = (a_F + b_F) \left( \sum_{e \in F} n_e - n \frac{d(F)}{k} \right). }
We define $\kappa = \kappa(G) = \min \{ (a_F + b_F) : \es \subsetneq F \subsetneq E \} > 0$.
We will need the following lemma.

\begin{lemma}
\label{lem:equiv-metric}
There exist constants $b = b(G) > 0$ and $B = B(G)$ such that for all $\bw \in \cK_G$ we have
\eqn{e:equiv-metric}
{ b \, \dist( \bw, \partial \cR_G )
  \le \min_{\es \subsetneq F \subsetneq E} \left\{ \sum_{e \in F} w_e - \frac{d(F)}{k} \right\}
  \le B \, \dist( \bw, \partial \cR_G ). }
We also have 
\eqn{e:equiv-metric2}
{ \frac{1}{2} L^{F,n}(n \bw)
  \le n \left( \sum_{e \in F} w_e - \frac{d(F)}{k} \right)
  \le \frac{1}{\kappa} L^{F,n}(n \bw), \quad n \ge 1. }
\end{lemma}

\begin{proof}
The proof of Lemma \ref{lem:prop-R_G}(ii) showed that $\cK_G \setminus \cR_G = \partial \cR_G$.
Therefore, if $\bw \in \cK_G \setminus \cR_G$ then $\sum_{e \in F} w_e = d(F)/k$ for some 
$\es \subsetneq F \subsetneq E$, and $\dist(\bw, \partial \cR_G) = 0$. In particular, the first 
statement of the lemma holds when $\bw \in \cK_G \setminus \cR_G$. Henceforth assume that 
$\bw \in \cR_G$. Then since $\partial \cR_G = \cup_{\es \subsetneq F \subsetneq E} H_F \cap \cK_G$,
we have
\eqn{e:dist-HF}
{ \dist( \bw, \partial \cR_G )
  = \min_{\es \subsetneq F \subsetneq E} \dist( \bw, \cK_G \cap H_F )
  \ge \min_{\es \subsetneq F \subsetneq E} \dist( \bw, H_E \cap H_F ). }
We claim that the last inequality is in fact an equality. Let $F$ be a set 
for which the minimum in the right hand side of \eqref{e:dist-HF} is attained.
Let $\bw_0$ be the orthogonal projection of $\bw$ onto $H_E \cap H_F$ in the linear
space $H_0$. If the line segment $\bw \, \bw_0$ had any interior point $\bw_1$ 
belonging to any other $H_{F'}$, then this would contradict the minimality 
of $F'$. Therefore, the entire line segment $\bw \, \bw_0$, apart from $\bw_0$,
belongs to $\cR_G$, with $\bw_0 \in \partial \cR_G$.
Hence $\dist ( \bw, H_E \cap H_F) = \dist( \bw, \bw_0 ) \ge \dist ( \bw, \partial \cR_G )$.
This proves our claim. Since $\bw \in H_E$, there exists a constant $B_0$, that only depends on 
$\min \{ \text{angle between $H_E$ and $H_F$} : \es \subsetneq F \subsetneq E \}$,
such that 
\eqnst
{ \dist( \bw, H_F ) 
  \le \dist( \bw, H_E \cap H_F ) 
  \le B_0 \, \dist ( \bw, H_F ). }
This implies the first statement of the lemma, since 
$\dist ( \bw, H_F ) = |F|^{-1/2} \left( \sum_{e \in F} w_e - \frac{d(F)}{k} \right)$.
The second statement of the lemma follows from the definition of $\kappa(G)$,
and the fact that $a_F, b_F \le 1$ (since $\bu^F$ is a unit vector).
\end{proof}

Recall that we write $\bn = n \bx$ for the starting state. Given a randomized
strategy, we write $\bX(t) = \frac{1}{n-t} \bN(t)$. 
Note that we allow the processes $\bN(t)$, $\bX(t)$, etc.~to have 
negative entries, and once this happens, we have $\bX(t) \not\in \cS_G$
for all further times.
We write $\bY(t-1)$ for the vector of edge weights that our strategy 
prescribes for round $t$, and $E(t) \in E$ for the random edge selected 
in round $t$ according to this strategy. We write 
\eqnst
{ \cF_t
  = \sigma \left( \bN(s),\, \bY(s) : 0 \le s \le t \right) }
for the filtration of the process.

\subsection{Steering}
\label{ssec:steering}

In the following proposition we show that if $n$ is large enough, 
then starting from any state in $\cR_G$ that is bounded away from the 
boundary, there is a strategy that steers the process close to 
any other such point in $\cR_G$.

\begin{proposition}
\label{prop:main-steer}
Given $\delta > 0$, there exist $c_1 = c_1(G,\delta) > 0$, 
$\lambda_1 = \lambda_1(G, \delta) > 0$, $n_0 = n_0 (G, \delta)$, 
$K_1 = K_1(G, \delta)$ and $C_1 = C_1(G, \delta)$ such that the following holds. 
Let $n$ and $n_1$ be any positive integers such that $n \ge ( 1 + K_1) n_1$ and
$n_1 \ge n_0$. Suppose that $\bn = n \bx$ with $\dist(\bx, \partial \cR_G) \ge \delta$. 
Suppose also that $\bz \in \cR_G$ with $\dist( \bz, \partial \cR_G ) \ge \delta$, 
with $n_1 \bz$ having integer coordinates. There exists a randomized 
strategy starting from state $\bn$ such that under this strategy we have:
\eqn{e:hit-exact}
{ \P [ \bN(n - n_1) = n_1 \bz ]
  \ge c_1; }
and for all $q \ge 1$ we have
\eqn{e:not-deviate}
{ \P \left[ \left| \bN(n - n_1) - n_1 \bz \right| > q \right]
  \le C_1 \exp ( - \lambda_1 q ). }
\end{proposition}

The strategy will be defined in three stages: in the first stage 
we reduce $| \bN(t) - (n - t) \bz|$ to $O(1)$; in the second stage we 
keep it within $O(1)$ until time $n - n_1 - O(1)$; and we use the 
last $O(1)$ steps to attempt to hit $n_1 \bz$ exactly. 
The first two of these steps are the content of the next two lemmas. 
After proving the lemmas we assemble them to prove 
Proposition \ref{prop:main-steer}.

\begin{lemma}
\label{lem:1st-stage}
Given $\delta > 0$ there exists $K_2 = K_2(G, \delta)$,
$d_0 = d_0(\delta)$, $\lambda_2 = \lambda_2(G, \delta) > 0$
and $C_2 = C_2(G)$ such that for any $\bx, \bz$ with 
$\dist(\bx, \partial \cR_G), \dist(\bz, \partial \cR_G) \ge \delta$
the following holds. For any $n, n'$ with $n \ge K_2 n'$ and
$n'$ large enough there is a randomized strategy starting from 
state $\bn = n \bx$ such that the stopping time
\eqnst
{ \tau_{d_0}
  = \inf \{ t \ge 0 : | \bN(t) - (n-t) \bz | \le d_0 \} }
satisfies 
\eqn{e:1st-stage-estimate}
{ \P [ \tau_{d_0} > n - n' ]
  \le C_2 \exp ( - \lambda_2 n' ). }
\end{lemma}

\begin{proof}
The value of $d_0 > 0$ will be chosen in course of the proof.
We are also going to use a small parameter $0 < \eps_0 < \delta/4$,
chosen later. The first step of the proof is to reach an
$\eps_0$-neighbourhood of $\bz$.

Let $\by$ be the point where the halfline starting at $\bz$ and 
passing through $\bx$ intersects $\partial \cR_G$. Let $\bu$ denote
the unit vector with the same direction as $\bx - \bz$.
In the first step, we use the following strategy: given the current state 
$\bN(t) = (n-t) \bX(t)$, we select $\bY(t) \in \partial\cR_G$ 
such that $\bY(t) - \bX(t)$ is a positive multiple of $\bu$.
In particular, $\bY(0) = \by$.
We employ this strategy until the stopping time $\tau(1)$ defined by
\eqnst
{ \tau(1)
  = \inf \{ t \ge 0 : | \bX(t) - \bz | \le \eps_0 \}. }
Let us write $\bX^\ort(t)$ for the component of the vector 
$\bX(t) - \bz$ orthogonal to $\bu$.
Let 
\eqn{e:S(t)-1}
{ S(t)
  = \langle \bN(t) - (n-t) \bz, \bu \rangle. }
Since 
\eqnst
{ \bN(t+1)
  = \left( \bN(t) - \bY(t) \right) + \left( \bY(t) - \bone^{E(t+1)} \right), }
and the second term has mean $\bzero$ given $\cF_t$, we have
\eqnspl{e:mart}
{ \E [ S(t+1) \,|\, \cF_t ]
  &= S(t) - \langle \bY(t) - \bz, \bu \rangle. }
Since $\bx$ and $\bz$ are bounded away from $\partial \cR_G$,
there exist $\mu = \mu(G,\delta) > 1$ and $\eps_0 = \eps_0(G, \delta) > 0$
such that as long as $|\bX^\ort(t)| \le \frac{\eps_0}{2}$, we have
\eqn{e:drift-lb}
{ \langle \bY(t) - \bz, \bu \rangle
  \ge \mu |\bx - \bz|. }
This implies that $S'(t) = S(t) + t \mu |\bx - \bz|$ is a
supermartingale as long as $|\bX^\ort(t)| \le \eps_0/2$.
On the other hand, due to the calculation in \eqref{e:drift1}, 
$\bX^\ort(t)$ is a martingale.

Let $t_1 = \frac{1 + \mu}{2 \mu} n$. Due to the choice of 
$\mu$ and $\eps_0$, we have the inclusions
\eqnspl{e:terminate}
{ \left\{ \tau(1) > t_1 \right\}
  &\subset \{ \text{$|\bX^\ort(s)| > \eps_0/2$ for some $0 \le s \le t_1$} \} \\
  &\qquad\qquad \cup \{ \text{$S(s) > \mu (n-s) |\bx - \bz|$ for some $0 \le s \le t_1$} \} 
     \cup \{ \text{$S(t_1) \ge 1$} \} \\
  &\subset \left\{ \max_{0 \le s \le t_1} |\bX^\ort(s)| > \eps_0/2 \right\} 
     \cup \left\{ \max_{0 \le s \le t_1} S'(s) - S'(0) > (\mu - 1) n |\bx - \bz| \right\} \\
  &\qquad\qquad \cup \left\{ \max_{0 \le s \le t_1} S'(s) - S'(0) 
    > \frac{\mu - 1}{2} n |\bx - \bz| \right\}. }
The inclusions \eqref{e:terminate} imply 
\eqnspl{e:1st-stage-bad-0}
{ \P [ \tau(1) > t_1 ]
  \le \P \left[ \max_{0 \le s \le t_1} S'(s) - S'(0) > \frac{\mu - 1}{2} n |\bx - \bz| \right]
      + \P \left[ \max_{0 \le s \le t_1} |\bX^\ort(s)| > \eps_0/2 \right]. }
Since $S'(t)$ has increments bounded by $(1 + \mu) \sqrt{2}$, 
while $|\bX^\ort(t+1) - \bX^\ort(t)| \le \sqrt{2}/(n-t-1)$, 
we can apply the Azuma-Hoeffding inequality 
(see \cite[Exercise E14.2]{Wbook} or \cite[Theorem 12.2(3)]{GSbook}) 
to $\{ S'(t) \}_{t \ge 0}$ as well as to the projection of 
$\{ \bX^\ort(t) \}_{t \ge 0}$ to each coordinate direction. 
This yields
\eqnspl{e:1st-stage-bad}
{ \P [ \tau(1) > t_1 ] 
  &\le \exp \left( - \frac{(\mu - 1)^2}{8} \frac{n^2 |\bx - \bz|^2}{t_1 \, 2 \, (1+\mu)^2} \right) 
      + 2 |E| \exp \left( - \frac{1}{8} \frac{\eps_0^2}{t_1 |E| \sum_{s = 1}^{t_1} \frac{2}{(n - s)^2} } \right) \\
  &\le C' \exp ( - \lambda'n ) }
for some $\lambda' = \lambda'(\mu, \eps_0) > 0$ and $C' = C'(G)$. 

For the second step we condition on the point
$\bn_1 = n_1 \bx_1 = \bN(\tau(1))$, such that $n - n_1 \le t_1$ and
$| \bx_1 - \bz | \le \eps_0 < \delta/4$. For ease of notation,
we re-parametrize time for this step so that $\bN(0) = \bn_1$.
We choose $\bY(t)$ to be the point where the halfline starting at $\bz$
and passing through $\bX(t)$ intersects $\partial \cR_G$. Let us write
$\bu(t)$ for the unit vector with the same direction as $\bX(t) - \bz$.
Decompose $\bX(t+1) - \bz = X'(t+1) \bu(t) + \bX''(t+1)$, where
$\langle \bX''(t+1), \bu(t) \rangle = 0$.
As long as $|\bN(t) - (n-t) \bz| \ge d_0$, we have 
\eqnsplst
{ | \bN(t+1) - (n-t-1) \bz|
  &= \sqrt{ \langle \bN(t+1) - (n-t-1) \bz, \bu(t) \rangle^2 + (n-t-1)^2|\bX''(t+1)|^2 } \\
  &\le \sqrt{ \langle \bN(t+1) - (n-t-1) \bz, \bu(t) \rangle^2 + 2 } \\
  &\le \langle \bN(t+1) - (n-t-1) \bz, \bu(t) \rangle + \frac{2}{d_0-\sqrt{2}}. }
Therefore, 
\eqnsplst
{ \E \big( | \bN(t+1) - (n-t-1) \bz| \,\big|\, \cF_t \big) 
  &\le \E \big( \langle \bN(t+1) - (n-t-1) \bz, \bu(t) \rangle \,\big|\, \cF_t \big) 
     + \frac{2}{d_0-\sqrt{2}} \\
  &= \langle \bN(t) - (n-t) \bz, \bu(t) \rangle - \langle \bY(t) - \bz, \bu(t) \rangle 
     + \frac{2}{d_0-\sqrt{2}} \\
  &\le |\bN(t) - (n-t) \bz| - \delta + \frac{2}{d_0-\sqrt{2}}. }
Hence if we require that $d_0 \ge \sqrt{2} + \frac{4}{\delta}$, then 
\eqnst
{ D(t)
  = |\bN(t) - (n-t) \bz| + \frac{\delta}{2} t, \quad t \ge 0, }
is a supermartingale until $\tau_{d_0}$. Since the increments of
$D(t)$ are bounded by $2 + \frac{\delta}{2} < 3$, and 
$\eps_0 < \frac{\delta}{4}$, it follows with $t_2 = \frac{3}{4} n_1$ that
\eqnst
{ \P \left[ \tau_{d_0} > t_2 \right]
  \le \P \left[ \max_{0 \le s \le t_2} (D(s) - D(0)) > \frac{\delta}{8} n_1 \right]
  \le \exp \left( - \frac{\delta^2 t_2^2}{64 \cdot 3^2 \, t_2} \right)
  \le \exp ( - \lambda'' n_1 ) }
with some $\lambda'' = \lambda''(\delta) > 0$. 

Putting the two parts together, the statement follows if we choose
$K_2 = \frac{8 \mu}{\mu - 1}$.
\end{proof}

\begin{lemma}
\label{lem:2nd-stage}
Given $\delta > 0$ there exist $\lambda_3 = \lambda_3(\delta) > 0$ 
and $C_3 = C_3(\delta)$ such that such that for all $n' \ge n'' \ge 0$ 
and all $\bw, \bz \in \cK_G$ with $\dist(\bz, \partial \cR_G) \ge \delta$,
$|n' \bw - n'\bz| \le d_0(\delta)$ the following holds. There 
exists a randomized strategy starting in state $\bn' = n' \bw$ such that 
for all $q \ge 1$ we have
\eqn{e:2nd-stage-estimate}
{ \P \left[ | \bN(n' - n'') - n'' \bz | > q \right]
  \le C_3 \exp ( - \lambda_3 q ). }
\end{lemma}

\begin{proof}
When $| \bN(t) - (n'-t) \bz| < d_0$, let us apply an arbitrary move,
otherwise, let us follow the strategy used in the second part of 
Lemma \ref{lem:1st-stage}. We saw in the proof of Lemma \ref{lem:1st-stage} 
that 
\eqnst
{ D(t)
  = |\bN(t) - (n-t) \bz| + \frac{\delta}{2} \sum_{0 \le s < t} 
    I [ |\bN(s) - (n-s) \bz| \ge d_0 ] }
is a supermartingale on any time interval $s \in [t_1,t_2)$ on which
$|\bN(s) - (n-s) \bz| \ge d_0$. Assume the event 
\eqnst
{ F(q)
  = \left\{ \text{$|\bN(n' - n'') - n'' \bz | > 4 q$} \right\}, }
and suppose $q > d_0$. When $n' - n'' < q$, the event $F(q)$ is 
impossible, because $| \bN(0) - n' \bz | \le d_0 < q$ and the 
increments of $|\bN(t) - (n-t) \bz|$ are bounded by $2$.
Hence we may assume that $\ell_\maxim := \lfloor (n' - n'')/q \rfloor \ge 1$. 
Since $D(0) \le d_0 < q$, the inequalities 
\eqn{e:not-all}
{ | \bN(n' - n'' - \ell q ) - (n'' + \ell q) \bz |
  > 4 q, \quad \ell = 0, \dots, \ell_\maxim, }
cannot all simultaneously be satisfied. Summing over the 
smallest $\ell$ for which \eqref{e:not-all} fails, we have 
\eqnspl{e:2nd-stage-bad-3}
{ \P [ F(q) ]
  &\le \sum_{1 \le \ell \le \ell_\maxim} \P \left[ D(n' - n'') 
     - D(n' - n'' - \ell q) > \frac{\delta}{2} q \ell \right] \\
  &\le \sum_{\ell \ge 1} \exp \left( - \frac{1}{8} \frac{\delta^2 q^2 \ell^2}{3^2 \, q \ell} \right) 
  \le C_3 \exp ( - \lambda_3 q ). }
Adjusting the constant $C_3$, if necessary, we have the statement for all
$q > 0$. This completes the proof.
\end{proof}

\begin{remark}
Note that the above strategy does not require the coordinates to stay positive.
This will become important in Section \ref{ssec:II}.
\end{remark}

\begin{proof}[Proof of Proposition \ref{prop:main-steer}.]
Observe that if there is no point $\bw$ such that 
$\dist( \bw, \partial \cR_G ) \ge \delta$, then the statement of the
Proposition holds vacuously. Henceforth assume that $\delta$ is small enough
so that the set above is non-empty. We choose $q_0 \ge 2$ so that for the event $F(q)$ introduced in the
proof of Lemma \ref{lem:2nd-stage} we have $\P [ F(q_0/4) ] \le \frac{1}{2}$.
Let $M$ be the smallest integer such that 
\eqnst
{ M 
  \ge \left( \min \left\{ w_e : e \in E,\, \bw \in \cR_G,\, 
    \dist (\bw, \partial \cR_G) \ge \delta \right\} \right)^{-1}, }
which is finite by our assumption on $\delta$.
We choose $K_1$ and $n_0$ such that $n \ge K_1 n_1$ and $n_1 \ge n_0$ imply
$n \ge K_2 ( n_1 + M q_0 )$, where $K_2$ is the constant from 
Lemma \ref{lem:1st-stage}.
Following the strategies in Lemmas \ref{lem:1st-stage} and \ref{lem:2nd-stage}
over the time interval $[n, n - n_1 - M q_0]$ we have
\eqn{e:2nd-stage-good}
{ \P \left[ | \bN(n - n_1 - M q_0) - (n_1 + M q_0 ) | \le q_0 \right]
  \ge \frac{1}{2} - C_2 \exp ( - \lambda_2 n_1 )
  \ge \frac{1}{4}, }
if $n_0$ is large enough. On the event in \eqref{e:2nd-stage-good} we have
\eqnsplst
{ &N_e(n - n_1 - M q_0) - n_1 z_e \\
  &\qquad \ge (M q_0) z_e - | N_e(n - n_1 - M q_0) 
     - (n_1 + M q_0) z_e | \\
  &\qquad \ge q_0 - q_0
  = 0, \quad e \in E. }
Therefore, $\bN(n - n_1 - M q_0) \ge n_1 \bz$ componentwise, and
there is a strictly positive probability $c_1 = c_1(G, \delta) > 0$
that $n_1 \bz$ can be hit exactly from the state $\bN(n - n_1 - M q_0)$. 
This proves \eqref{e:hit-exact} of the Proposition.
Since the form of the bound \eqref{e:2nd-stage-estimate} is not
affected by taking $M q_0$ extra steps, statement 
\eqref{e:not-deviate} follows from the estimates \eqref{e:1st-stage-estimate}
and \eqref{e:2nd-stage-estimate} of Lemmas \ref{lem:1st-stage} and \ref{lem:2nd-stage}.
\end{proof}

\subsection{Proof of the Main Theorem}
\label{ssec:proof-main}

In this section we complete the proof of Theorem \ref{thm:phase-trans-graph}.

\begin{proof}[Proof of Theorem \ref{thm:phase-trans-graph}(i).] 
Fix $\bx \in \cI_G$, and let 
$\es \subsetneq F \subsetneq E$ be a set such that
$\sum_{e \in F} x_e < \frac{d(F)}{k}$. Then for some $\eps = \eps(G, \bx) > 0$
and sufficiently large $n$ we have 
$\frac{1}{n} \sum_{e \in F} N_e(0) < \frac{d(F)}{k} - \eps$.
Let 
\eqnst
{ Y_t
  = \begin{cases}
    1 & \text{if $V_t = v$ and $\deg_F(v) = \deg_G(v)$;} \\
    0 & \text{otherwise.} 
    \end{cases} }
Since any $v$ with $\deg_F(v) = \deg_G(v)$ must be
assigned to one of the edges in $F$, we have
\eqnsplst
{ p_G(\bn)
  &\le \P \left[ \sum_{t = 1}^n Y_t \le \sum_{e \in F} N_e(0) \right]
  \le \P \left[ \frac{1}{n} \sum_{t=1}^n Y_t < \frac{d(F)}{k} - \eps \right] 
  \le \exp \left( -n \frac{\eps^2}{4} \right). }
by Bernstein's inequality; see \cite[Theorem 2.2(1)]{GSbook}.
The rate of decay is bounded away from $0$ as long as 
$\bx$ is bounded away from $\partial \cR_G$.
\end{proof}

\begin{proof}[Proof of Theorem \ref{thm:phase-trans-graph}(ii).] 
We show that for any fixed $\delta > 0$ we have 
\eqn{e:limits}
{ \lim_{n \to \infty} M_n
  = \lim_{n \to \infty} m_n 
  = \alpha, } 
where
\eqnsplst
{ m_n
  &= m_n(\delta)
  = \min \left\{ p_G(\bn) : \sum_{e \in E} n_e = n,\, 
    \dist(\bn/n, \partial \cR_G) \ge \delta \right\}, \quad n \ge 1; \\
  M_n
  &= M_n(\delta) 
  = \max \left\{ p_G(\bn) : \sum_{e \in E} n_e = n,\, 
    \dist(\bn/n, \partial \cR_G) \ge \delta \right\}, \quad n \ge 1; \\
  \alpha
  &= \alpha(\delta) 
  = \liminf_{n \to \infty} m_n(\delta). }
We consider $n' \ge n_0$, $n \ge K_1 n'$ and $\bn = n \bx$ such that 
$m_n = p_G(\bn)$. We apply Proposition \ref{prop:main-steer} with 
$\bz = \bn'/n'$, where $\bn'$ is chosen so that $M_{n'} = p_G(\bn')$.

Let $\varphi(\br)$ denote the probability that with the
strategy described in Proposition \ref{prop:main-steer} 
the state at time $n - n'$ is $n' \bz + \br$,
where $\sum_{e \in E} r_e = 0$. Due to 
Proposition \ref{prop:main-steer}, we have
$\varphi(0) \ge c_1$. Therefore, we can write
\eqnsplst
{ m_n
  = p_G(\bn) 
  &\ge \sum_{\br : \sum_{e \in E} r_e = 0} \varphi(\br) \, p_G(n' \bz + \br) 
  \ge c_1 p_G(n' \bz) + \sum_{\substack{\br \not= \bzero : \\ \sum_{e \in E} r_e = 0}} 
     \varphi(\br) \, p_G(n' \bz + \br) \\
  &\ge c_1 (M_{n'} - m_{n'}) + \sum_{\br : \sum_{e \in E} r_e = 0}
     \varphi(\br) \, m_{n'} \\
  &\ge c_1 (M_{n'} - m_{n'}) + m_{n'} - C \exp ( -\lambda n' ) }
with some $\lambda > 0$ and $C$ depending on $\delta$ and 
$\lambda_1, \lambda_2, \lambda_3$. Rearranging gives
\eqn{e:with-c}
{ M_{n'} - m_{n'}
  \le \frac{1}{c_1} ( m_n - m_{n'} ) + \frac{C}{c_1} \exp ( -\lambda n' ). }
Since $n \ge K n'$ was arbitrary, taking $\liminf_{n \to \infty}$ yields
\eqn{e:liminf}
{ M_{n'} - m_{n'}
  \le \frac{1}{c_1} (\alpha - m_{n'}) + \frac{C}{c_1} \exp ( - \lambda n' ). }
Taking $\limsup_{n' \to \infty}$ in \eqref{e:liminf} yields $M_{n'} - m_{n'} \to 0$.
Taking $\liminf_{n' \to \infty}$ in \eqref{e:liminf} yields
\eqnst
{ 0 
  \le \liminf_{n' \to \infty} (M_{n'} - m_{n'})
  \le \frac{1}{c_1} (\alpha - \limsup_{n' \to \infty} m_{n'})
  \le 0. }
This shows that $\lim_{n' \to \infty} m_{n'} = \alpha$, and the 
proof of \eqref{e:limits} is complete.

The limit does not depend on $\delta$, since for
$0 < \delta_1 < \delta_2$ we have
\eqnst
{ m_n(\delta_1)
  \le m_n(\delta_2)
  \le M_n(\delta_2)
  \le M_n(\delta_2), }
and hence $\alpha(\delta_1) = \alpha(\delta_2) = c_G$.

We conclude the proof by noting that $c_G > 0$. This is because
Proposition \ref{prop:main-steer} implies that 
the process can be steered close to the point $n_0 \bx^*$
for a sufficiently large $n_0$ with positive probability, and
from here there is a strictly positive probability of winning.
\end{proof}

\begin{remark}
Since the left hand side of \eqref{e:liminf} is non-negative,
we can rearrange to get
\eqnst
{ m_{n'} 
  \le \alpha + C \exp ( - \lambda n' ), \quad n' \ge n_0. }
We do not have a corresponding exponential lower bound on the speed
at which the limit $\alpha$ is approached. 
See Question \ref{prob:exp-converge} in Section \ref{sec:open}.
\end{remark}

\section{Upper bounds in the critical region}
\label{sec:critical-proof}

In this section we obtain estimates in the critical region.
This requires distinguishing a few cases that we state as 
separate propositions in the next section, and use them to prove
Theorem \ref{thm:bdry}. The proofs of the three propositions are given
in Sections \ref{ssec:I}, \ref{ssec:II} and \ref{ssec:III}, 
respectively.

\subsection{Statements of upper bounds in three subregions}
\label{ssec:state-regions}

We define the sets of configurations
\eqnspl{e:cB_G}
{ \cB_G^I(n; A)
  &= \left\{ \bn \in n \cS_G : \, \text{for some 
     $\es \subsetneq F \subsetneq E$ we have 
     $L^{F,n}(\bn) \le -A \sqrt{n}$} \right\} \\   
  \cB_G^{II}(n; A)
  &= \left\{ \bn \in n \cS_G : \, \text{for all
     $F$ with $0 < d(F) < k$ we have 
     $L^{F,n}(\bn) \ge A \sqrt{n}$} \right\} \\ 
  \cB_G^{III}(n; A)
  &= \left\{ \bn \in n \cS_G : \, 
     -A \sqrt{n} < \min_{F: 0 < d(F) < k} L^{F,n}(\bn) 
     < A \sqrt{n} \right\}. }

\begin{proposition}
\label{prop:far-enough}
For all $A > 0$ we have
\eqnst
{ \limsup_{n \to \infty} \, \max \{ p_G(\bn) : \bn \in \cB_G^I(n; A) \}
  \le \exp \left( - \frac{A^2}{8} \right). } 
In particular, the $\limsup$ is at most $c_G$, if
$A \ge \sqrt{8 \log (1/c_G)}$.
\end{proposition}

\begin{proposition}
\label{prop:in-enough}
There exist constants $C_4 = C_4(G)$ and $\lambda_4 = \lambda_4(G) > 0$
such that for all $A \ge 1$ we have 
\eqn{e:in-enough}
{ \limsup_{n \to \infty} \, \max \{ p_G(\bn) : \bn \in \cB_G^{II}(n; A) \}
  \le c_G + C_4 \exp ( - \lambda_4 A^2 ). }
\end{proposition}

\begin{proposition}
\label{prop:near-critical}
There exists $A_0 = A_0(G)$ such that for all $A \ge A_0$ we have
\eqnst
{ \limsup_{n \to \infty} \, \max \{ p_G(\bn) : \bn \in \cB_G^{III}(n; A) \}
  \le c_G + C_4 \exp ( - \lambda_4 A^2 ). }
\end{proposition}

\begin{proof}[Proof of Theorem \ref{thm:bdry} assuming 
Propositions \ref{prop:far-enough},
\ref{prop:in-enough}, \ref{prop:near-critical}.]
Given $\eps > 0$, choose $A$ sufficiently large so that 
each of the upper bounds in Propositions \ref{prop:far-enough},
\ref{prop:in-enough} and \ref{prop:near-critical}
is at most $c_G + \eps$. Since with this fixed choice of $A$ 
the sets $\cB_G^I$, $\cB_G^{II}$ and $\cB_G^{III}$ 
cover all possibilities, the statement follows.
\end{proof}

\subsection{Upper bound for $\cB_G^{I}$}
\label{ssec:I}

\begin{proof}[Proof of Proposition \ref{prop:far-enough}.]
We may fix the set $F$ in the definition of $\cB_G^I(n; A)$
and argue separately for each such set. Let us fix $\delta > 0$.
Due to Theorem \ref{thm:phase-trans-graph}(i), 
we may restrict to $\bn$ such that 
\eqnst
{ - \delta n 
  < L^{F,n}(\bn) 
  \le - A \sqrt{n}. }
Let us follow the optimal strategy starting in
configuration $\bn$. The process $S(t) = L^{F,n-t} (\bN(t))$
is a supermartingale due to
\eqn{e:S(t)-supermart}
{ \E [ S(t+1) \,|\, \cF_t ]
  = S(t) - \langle \bY(t) - \bz^F, \bu^F \rangle
  \le S(t). }
Consider the stopping time
\eqnst
{ \tau
  = \left( \lfloor n - c \sqrt{n} \rfloor + 1 \right) \, \wedge \, 
    \inf \{ t \ge 0 : S(t) < -\delta(n-t) \}, }
where $c = \frac{A}{2 \delta}$. Then we have
\eqnsplst
{ \P [ \tau > n - c \sqrt{n} ]
  &\le \P \left[ \max_{0 \le t \le \lfloor n - c\sqrt{n} \rfloor} S(t) - S(0) > (A - \delta c ) \sqrt n \right] \\
  &\le \exp \left( -\frac{1}{2} \frac{(A - \delta c)^2 \, n}{\lfloor n - c \sqrt{n} \rfloor} \right)
  \le \exp \left( - \frac{A^2}{8} \right). }
Due to the optimality equation, $p_G(\bN(t))$ is a bounded martingale.
Hence by optional stopping we have
\eqnspl{e:OST}
{ p_G(\bn)
  &= \E [ p_G(\bN(\tau)) ;\,  \tau \le n - c \sqrt{n},\, S(\tau) < - \delta (n - \tau) ]
    + \E [ p_G(\bN(\tau)) ;\, \tau > n - c \sqrt{n} ], }
The first term in the right hand side of \eqref{e:OST} is at most
\eqnst
{ \max \left\{ p_G(\bn') : \| \bn' \|_1 \ge c \sqrt{n},\, L^{F,n'}(\bn') < -\delta n' \right\}, }
which goes to $0$, as $n \to \infty$, due to Theorem \ref{thm:phase-trans-graph}(i). 
The second term in the right hand side of \eqref{e:OST} is
at most $\P [ \tau > n - c \sqrt{n} ] \le \exp ( - \frac{A^2}{8} ) < c_G$, due to our choice
of $A$. This completes the proof of the Proposition.
\end{proof}

\subsection{Upper bound for $\cB_G^{II}$}
\label{ssec:II}

We start with two propositions that strengthen 
Proposition \ref{prop:main-steer}, and will be used in the proof of
Proposition \ref{prop:in-enough}. In the first, we give a lower
bound on the probability that the process can be steered away 
from the boundary, if at least order $\sqrt{n}$ away.

\begin{proposition}
\label{prop:strengthened}
There exist $\lambda_5 = \lambda_5(G) > 0$, $\gamma = \gamma(G) > 0$, 
$c_5 = c_5(G)$, $C_5 = C_5(G)$ and $n'_0 = n'_0(G)$ such that for 
all $A \ge 1$ the following holds. Let $n, n'$ satisfy
$n^{\gamma} \ge n' \ge n'_0$, and let $\bn = n \bx$ be a
configuration such that 
\eqnspl{e:assump}
{ \sum_{e \in F} x_e 
  &\ge \frac{1}{k} d(F) + \frac{A}{\sqrt{n}}, \quad 
    \text{for all $\es \subsetneq F \subsetneq E$.} }
There exists a randomized strategy starting from $\bn$ such that 
for the stopping time 
\eqnst
{ \tau 
  = \inf \{ t \ge 0 : \dist(\bX(t), \partial \cR_G) \ge c_5 \} }
we have
\eqnst
{ \P [ \tau > n - n' ]
  \le C_5 \exp ( -\lambda_5 A^2 ). }
\end{proposition}

\begin{proof}
Let $\by$ be the point where the halfline
starting at $\bx^*$ and passing through $\bx$ intersects $\partial \cR_G$. 
Write $d = | \bx - \by |$, and note that $d \ge \frac{A}{B} \frac{1}{\sqrt{n}}$,
due to Lemma \ref{lem:equiv-metric}.
Let $r$ be the smallest integer such that $(3/2)^r d \ge \frac{1}{2} | \bx^* - \by |$.
We fix a small number $\eta > 0$ such that $\frac{1}{2} - \eta > \frac{4}{9}$.
Then it is straightforward to check that the choice of $r$ ensures that
there exists $0 < \gamma = \gamma(G) < 1$ such that 
$(\frac{1}{2} - \eta)^r n \ge n^{\gamma}$, if $n \ge n_0$ for 
some $n_0 = n_0(G)$.

Consider the sequence of points $\bx = \by(0), \by(1), \dots, \by(r)$ defined by
\eqnst
{ \by(i)
  = \by + (3/2)^i (\bx - \by), \quad i = 0, 1, \dots, r. }
The following statement can be proved in essentially the same way as 
Lemma \ref{lem:1st-stage}. For $\eps > 0$ sufficiently small, there 
exists $\lambda = \lambda(G,\eta,\eps) > 0$ such that given any point $\bw \in \cR_G$
with $|\bw - \by(i)| < \eps (3/2)^i d$ and any $n$ such that 
$(\frac{1}{2} - \eta) n \ge n_0$ the following holds. 
There exists a randomized strategy starting in state $n \bw$ such that 
for the stopping time 
\eqnst
{ \tau(i)
  = \inf \{ t \ge 0 : |\bX(t) - \by(i+1)| < \eps (3/2)^{i+1} d \} }
we have
\eqnst
{ \P \left[ \tau(i) > \left( \frac{1}{2} + \eta \right) n \right] 
  \le \exp \left( - \lambda (3/2)^{2i} A^2 \right). }
Summing the upper bounds on $\tau(0), \tau(1), \dots, \tau(r-1)$ we obtain
that there is a randomized strategy starting from state $\bn$ such that 
for the stopping time 
\eqnst
{ \tau'
  = \inf \{ t \ge 0 : |\bX(t) - \by(r)| < \eps (3/2)^{r} d \} }
we have
\eqnst
{ \P [ \tau' > n - n^{\gamma} ]
  \le C \exp ( - \lambda A^2 ). }
Due to the choice of $r$, and for a sufficiently small $\eps$,
the point $\bX(\tau')$ is at least a fixed positive distance $c_5$ 
from $\partial \cR_G$, and hence $\tau \le \tau'$. This completes
the proof.
\end{proof}

The next proposition extends the result of Proposition \ref{prop:main-steer} 
to the case when the target state is anywhere in $\cK_G$.

\begin{proposition}
\label{prop:hit-any}
Given $\delta > 0$, there exists $\lambda_6 = \lambda_6(G) > 0$, $C_6 = C_6(G)$, 
$c_6 = c_6(G) > 0$, $K_6 = K_6(G,\delta)$ and $n_6 = n_6(G,\delta)$ such that for any 
$n_1 \ge K_6 n'$, $n' \ge n_6$ and configurations $\bn_1 = n_1 \bx$, $\bx \in \cR_G$,
$\dist(\bx, \partial \cR_G) \ge \delta$ and $\bn' = n' \bz$, $\bz \in \cK_G$
the following holds. There exists a randomized strategy starting in state 
$\bn_1$ such that 
\eqn{e:exactly2}
{ \P \left[ \bN(n_1 - n') = \bn' \right]
  \ge c_6, }
and
\eqn{e:not-deviate2}
{ \P \left[ | \bN(n_1 - n') - \bn' | > q \right]
  \le C_6 \exp ( - \lambda_6 q ), \quad q > 0. }
\end{proposition}

\begin{proof}
We consider the following intermediate point:
\eqnsplst
{ \bx''
  = \frac{1}{2} \bx + \frac{1}{2} \bx' \qquad \text{ and } \qquad
  \bn''
  = n' \bx + \bn' + O(1), }
where the $O(1)$ term guarantees that $\bn''$ has integer coordinates. 
Observe that $\dist( \bx'', \partial \cR_G)$
is at least a positive constant. Due to Proposition \ref{prop:main-steer}
we can steer the process from $\bn_1$ to a $(\delta/4)$-neighbourhood 
of $\bx''$ with probability at least $1 - C_1 \exp ( - \lambda_1 n' )$,
provided $K_6 \ge 2 K_1(G,\delta)$.
Let us call the point reached this way $(2 n') \by''$. Since
\eqnst
{ \by''
  = \bx'' + (\by'' - \bx'') 
  = \frac{1}{2} (\bx - 2(\by'' - \bx'')) + \frac{1}{2} \bx', }
and $| 2 (\by'' - \bx'') | < \frac{\delta}{2}$,
the point $\bw = \bx - 2(\by'' - \bx'')$ satisfies
$\dist( \bw, \partial \cR_G ) \ge \frac{\delta}{2}$. 

Now consider the steps of the strategy of Lemma \ref{lem:2nd-stage}
for the starting state $n' \bw$ and target state $0 \bw$ over
the time interval $[0,n' - M q_0]$, where 
$M \ge ( \min \{ w_e : e \in E \} )^{-1}$, and $q_0$ 
is chosen so that $F(q_0/4) \ge \frac{1}{2}$. 
Let $\widetilde{\bN}(t)$, $t \ge 0$ denote this process. 
If the coordinates do stay positive until time $n' - M q_0$,
there is a strictly positive probability of hitting state $\bzero$.
When $\bzero$ is not hit exactly, we have the bound
\eqnst
{ \P [ |\widetilde{\bN}(n')| > q ]
  = \P [ |\widetilde{\bN}(n') - \bzero| > q ]
  \le C_2 \exp ( - \lambda_2 q ). }
If we now apply exactly the same moves to the configuration 
$(2 n') \by''$, we obtain that the process 
$\bN(t) = \bn' + \widetilde{\bN}(t)$ hits
$\bn' = n' \bx'$ with positive probability, and satisfies the 
bound in \eqref{e:not-deviate2}.
\end{proof}

Since the proof of Proposition \ref{prop:in-enough} is quite long,
we first give a brief outline. Suppose we can select configurations 
$\bn$ and $\bn(\ell), \dots, \bn(1)$ in such a way that:\\
(a) $\bn/n$ is bounded away from $\partial \cR_G$, so that 
we have $p_G(\bn) \le c_G + \eps$;\\
(b) $\bn(\ell), \dots, \bn(1)$ are in the respective sets $\cB_G^{II}$
with each $p_G(\bn(i))$ close to the $\limsup$ in \eqref{e:in-enough};\\
(c) We can steer the process as follows:
$\bn \to \bn(\ell) \to \bn(\ell-1) \to \dots \to \bn(1)$;\\
(d) In each steering step we hit the target exactly with probability
bounded away from $0$.\\
If $\ell$ is large, step (d) ensures that $p_G(\bn)$ cannot be 
much smaller than the smallest of the $p_G(\bn(i))$'s, and the claim will follow.
The crux of the proof is parts (c)--(d), which rely on 
Propositions \ref{prop:strengthened} and \ref{prop:hit-any}. 
The argument is somewhat delicate, since the $\bn(i)$'s now can be
arbitrarily close to $\partial \cR_G$;
recall the definition of $\cB_G^{II}$ in \eqref{e:cB_G}. 
Therefore, Propositions \ref{prop:strengthened} and \ref{prop:hit-any}
will be applied on a suitable subgraph that omits some edges. 

We carry out the plan (a)--(d). We start with some preliminaries.
The first step is to subdivide $\cB_G^{II}$ according to which part 
of $\partial \cR_G$ is close. Given $\bn \in \cB_G^{II}$, let 
\eqnsplst
{ \cG
  = \cG(\bn ; G, A) 
  = \left\{ F \subset E : 
    L^{F,n}(\bn) < \frac{\kappa A}{2^{|E|+1}} \sqrt{n} \right\} 
  \qquad \text{ and } \qquad
  \overline{F}
  = \cup \cG, }
where $\kappa$ is the constant from Lemma \ref{lem:equiv-metric}.
It may so happen that $\overline{F} = \es$, in which case the arguments
we have to make are similar to and simpler than when $\overline{F} \not= \es$.
We will not spell out such arguments. Note that $F \in \cG$ implies 
$d(F) = 0$, since $\bn \in \cB_G^{II}$. Hence we have
\eqn{e:sum-e-bound}
{ \sum_{e \in \overline{F}} n_e 
  \le \sum_{F \in \cG} \sum_{e \in F} n_e
  \le \sum_{F \in \cG} \frac{1}{2 \kappa} L^{F,n}(\bn) 
  < \frac{1}{2} A \sqrt{n}. }
This implies $d(\overline{F}) = 0$, for $n$ large enough. Note that any $F$ with 
$d(F) = 0$ that is not contained entirely inside $\overline{F}$ satisfies
\eqnst
{ \sum_{e \in F} n_e 
  \ge \frac{1}{2} L^{F,n}(\bn)
  \ge \frac{\kappa A}{2^{|E|+2}} \sqrt{n}. }
Let us abreviate $\kappa_0 = \kappa / 2^{|E|+2}$.
In the remainder of this section, we are going to fix a possible
value $F_0$ of $\overline{F}$, and argue separately for each $F_0$.
With this in mind we make the following definitions. For any 
$F_0$ such that $d(F_0) = 0$, let
\eqnspl{e:max-def}
{ \cB_G^{II}(n; A, F_0)
  &= \left\{ \bn \in \cB_G^{II}(n; A) : \, 
    \parbox{8cm}{$\sum_{e \in F_0} n_e < \frac{1}{2} A \sqrt{n}$,
    and for all $F$ not contained in $F_0$ we have  
    $\sum_{e \in F} n_e - \frac{n}{k} d(F) \ge \kappa_0 A \sqrt{n}$} \right\} \\ 
  M_n(F_0)
  &= \max \left\{ p_G(\bn) : \, \bn \in \cB_G^{II}(n; A, F_0) \right\} \\
  \beta
  &= \limsup_{n \to \infty} M_n(F_0). }
Our task is to show that $\beta \le c_G + C \exp ( - \lambda A^2 )$
for each $F_0$ such that $\cB_G^{II}(n; A, F_0)$ is non-empty.

We will need to work on subgraphs of the form $G^{H} = (V, E^{H})$,
where $E^{H} = E \setminus H$, $H \subset F_0$. We write $\bn^H$ for 
the restriction of $\bn$ to $G^H$, that is: $\bn^H = (n_e : e \in E^H)$.
When no confusion can arise, we will write $n^H = \sum_{e \in E^H} n_e$.

\begin{lemma}
If $\cB_G^{II}(n; A, F_0)$ is non-empty, then 
for any $H \subset F_0$ the graph $G^{H}$ is connected.
\end{lemma}

\begin{proof} 
It is enough to consider $H = F_0$. Should $G^{F_0}$ not be 
connected, we could write $E = E_1 \cup F_0 \cup E_2$
as a disjoint union, where $E_1$ and $E_2$ are non-empty and 
do not share any vertex. Then we have 
$0 < d(E_1 \cup F_0), d(E_2 \cup F_0) < k$ and 
$d(E_1 \cup F_0) + d(E_2 \cup F_0) \ge k$. Therefore, 
if $\bn \in \cB_G^{II}(n; A, F_0)$, we have
\eqnsplst
{ \sum_{e \in E} n_e
  &= \sum_{e \in E_1 \cup F_0} n_e + \sum_{e \in E_2 \cup F_0} n_e 
    - \sum_{e \in F_0} n_e \\
  &\ge \frac{n}{k} d(E_1 \cup F_0) + \frac{1}{2} A \sqrt{n} 
      + \frac{n}{k} d(E_2 \cup F_0) + \frac{1}{2} A \sqrt{n}
    - \frac{1}{2} A \sqrt{n} \\
  &\ge n + \frac{1}{2} A \sqrt{n}
  > n, }
a contradiction.
\end{proof}

\begin{lemma}
\label{lem:str-applies} 
Let $H \subset F_0$ and $\bn \in \cB_G^{II}(n; A, F_0)$.\\
(i) We have $\bn^H/n^H \in \cK_{G^H}$.\\
(ii) Suppose in addition that $n_e \ge c A \sqrt{n}$
for all $e \in F_0 \setminus H$, with some $c > 0$. 
Then $\bn^H$ satisfies the assumption on the starting state of 
Proposition \ref{prop:strengthened}, with $A$ replaced by 
$\min\{ c A, \kappa_0 A \}$.
\end{lemma}

\begin{proof}
Both statements will be proved by the same computations.
Let $\es \subsetneq F \subsetneq (E \setminus H)$.
Since $d(H) \le d(F_0) = 0$, we have $d(F \cup H ; G) = d(F ; G^{H})$.
When this common value is $\ge 1$, we have
\eqnspl{e:restr-ge1}
{ \sum_{e \in F} n_e
  &\ge \sum_{e \in F \cup H} n_e - \frac{1}{2} A \sqrt{n}
  \ge \frac{n}{k} d(F \cup H ; G) + A \sqrt{n} - \frac{1}{2} A \sqrt{n} \\
  &\ge \frac{n^{H}}{k} d(F ; G^{H}) + \frac{1}{2} A \sqrt{n^H} 
  \ge \frac{n^{H}}{k} d(F ; G^{H}). }
This already suffices for part (i). When $d(F \cup H ; G) = d(F; G^H) = 0$ 
and $F$ is not a subset of $F_0$, we have
\eqn{e:restr-=0-1}
{ \sum_{e \in F} n_e 
  \ge \kappa_0 A \sqrt{n}
  \ge \kappa_0 A \sqrt{n^H}. }
When $\es \subsetneq F \subset F_0 \setminus H$, under the assumption
made in part (ii) we have
\eqn{e:restr-=0-2}
{ \sum_{e \in F} n_e
  \ge c A \sqrt{n}
  \ge c A \sqrt{n^H}. }
The three cases \eqref{e:restr-ge1}, \eqref{e:restr-=0-1} and
\eqref{e:restr-=0-2} complete the proof of part (ii).
\end{proof}

The main technical difficulty in the proof of 
Proposition \ref{prop:in-enough} is that we have no
control over how small $n_e(i)$ can get for $e \in F_0$, and therefore
these coordinates \emph{must} be hit exactly at each stage. 
We can do this, if the difference $n_e(i+1) - n_e(i) \ge 0$
is sufficiently small so that we have enough opportunity to 
play these edges (once the exact value is achieved,
we can ignore any such edge, since $d(F_0) = 0$.
The configurations introduced next will help us
overcome this technical difficulty.

Let $\bx^{*,F_0}$ denote the configuration 
introduced in \eqref{e:x^*}, with the graph $G$ replaced by $G^{F_0}$. 
Given $\delta > 0$ and $H \subsetneq F_0$, let 
\eqnst
{ \by^{*,F_0}(\delta; H) 
  = (1 - \delta) \bx^{*,F_0} 
    + \delta \frac{1}{|F_0 \setminus H|} \sum_{e \in F_0 \setminus H} \bone^e, }
where all vectors are regarded as being in $\R^{E^H}$.
Let $\bn^{*,F_0}(H) = n \by^{*,F_0}(\delta ; H) + O(1)$.

\begin{lemma}
\label{lem:x^*-mod} \ \\
(i) We have $\bx^{*,F_0} \in \cK_{G^H}$.\\
(ii) For all sufficiently small $\delta > 0$ we have 
$\by^{*,F_0}(\delta; H) \in \cR_{G^H}$ and 
$\dist(\by^{*,F_0}(\delta; H), \partial \cR_{G^H}) \ge \delta (B |F_0 \setminus H|)^{-1}$.\\
(iii) There exists $c_7(G) > 0$ such that for all sufficiently small $\delta > 0$
and all $\es \subsetneq F \subsetneq E^{F_0}$ we have
\eqnst
{ \left( \sum_{e \in E^{F_0}} n^{*,F_0}_e(H) \right)^{-1} 
    \sum_{e \in F} n^{*,F_0}_e(H) 
  \ge \frac{d(F ; G^{F_0})}{k} + c_7. }
\end{lemma}

\begin{proof}
(i) Let $\es \subsetneq F \subsetneq E^H$. We first consider the case 
when $F \not\subset F_0 \setminus H$ and 
$E \setminus F_0 \not\subset F$. Then we have
\eqn{e:not-subset1}
{ \sum_{e \in F} x^{*,F_0}_e
  = \sum_{e \in F \setminus F_0} x^{*,F_0}_e
  > \frac{d(F \setminus F_0 ; G^{F_0})}{k}
  = \frac{d(F \cup (F_0 \setminus H) ; G^H)}{k}
  \ge \frac{d(F ; G^H)}{k}. }
When $F \not\subset F_0 \setminus H$ and $E \setminus F_0 \subset F$, we have instead
\eqn{e:not-subset2}
{ \sum_{e \in F} x^{*,F_0}_e
  = \sum_{e \in F \setminus F_0} x^{*,F_0}_e
  = 1 
  > \frac{d(F ; G^H)}{k}. }
If $\es \subsetneq F \subset F_0 \setminus H$, we have
\eqn{e:subset}
{ \sum_{e \in F} x^{*,F_0}_e
  = 0 
  = \frac{d(F ; G^H)}{k}. }
This completes the proof of part (i).

(ii) If $\delta$ is sufficiently small, the inequalities \eqref{e:not-subset1}
and \eqref{e:not-subset2}, with $\bx^{*,F_0}$ replaced by $\by^{*,F_0}(\delta; H)$, 
remain strict. Also, Eqn.~\eqref{e:subset} becomes a strict inequality. 
The lower bound on the distance follows from Lemma \ref{lem:equiv-metric}. 

(iii) This follows from \eqref{e:not-subset1}, since the normalization
factor in the front is $[n (1 - O(\delta))]^{-1}$.
\end{proof}

\begin{proof}[Proof of Proposition \ref{prop:in-enough}.]
Given $\eps > 0$, we select a subsequence along which 
$M_n(F_0) > \beta - \eps$. For each $n$ in the subsequence,
select $\bn \in \cB_G^{II}(n, F_0)$ such that  
$p_G(\bn) > \beta - \eps$. By passing to a further subsequence,
we may assume that for each $e \in F_0$ the coordinates
$n_e$ are nondecreasing along the subsequence.

We now choose $\bn(1), \dots, \bn(\ell)$ and $\bn$.
Let $n(1) < \dots < n(\ell)$ and let $\bn(i) \in \cB_G^{II}(n(i); F_0)$,
$i = 1, \dots, \ell$, be a sequence of points such that:\\
(i) $n(i+1) \ge 2 (2 K_6 n(i))^{1/\gamma}$, $i = 1, \dots, \ell-1$;\\
(ii) $n_e(i+1) \ge n_e(i)$, for all $e \in F_0$, 
$i = 1, \dots, \ell-1$;\\
(iii) $p_G(\bn(i)) \ge \beta - \eps$, $i = 1, \dots, \ell$.\\
We further define $\bn$ in the following way. 
Let $n = 2 K_6 n(\ell)$, where $K_6$ is the constant of 
Proposition \ref{prop:hit-any}, and let 
$\bn = K_6 \, n(\ell) \, \by^{*,F_0}(\delta_1; \es) + K_6 \, \bn(\ell) + O(1)$ 
for a small $\delta_1 > 0$ for which the conclusions of 
Lemma \ref{lem:x^*-mod}(ii)--(iii) hold.
We will need that for all $e \in F_0$ we have
\eqn{e:F_0-bound}
{ n_e
  \le K_6 \, n(\ell) \, \frac{\delta_1}{|F_0|} 
     + K_6 \, \frac{1}{2} \, A \, \sqrt{n(\ell)} + O(1)
  < 2 \delta_1 K_6 \, n(\ell)
  = \delta_1 n, }
if $n(\ell)$ is large enough.
Also note that an application of Theorem \ref{thm:phase-trans-graph}(ii) 
yields $p_G(\bn) < c_G + \eps$.

We now define the strategy to steer from $\bn$ towards $\bn(\ell)$.
We first employ a strategy that plays an edge $e \in F_0$
with $N_e(t) > n_e(\ell)$, whenever that is possible, but never plays 
an edge $e \in F_0$ with $N_e(t) = n_e(\ell)$. We stop the first 
time $t$ when for all $e \in F_0$ we have $N_e(t) = n_e(\ell)$.
Such a strategy exists, since $d(F_0) = 0$. Since we start
with $N_e(0) - n_e(\ell) \le \delta_1 n$ (recall \eqref{e:F_0-bound}), 
if $\delta_1$ is sufficiently small, there is probability 
$\ge 1 - \exp ( - \lambda n )$ that we stop before time 
$C \delta n$ for some $C = C(G)$ and $\lambda > 0$. 
Moreover, the value on every edge is decreased 
by an amount at most $C \delta n$, and therefore it follows 
from Lemma \ref{lem:x^*-mod}(iii) that the configuration 
$\bn'$ reached has the property that  
$(\bn')^{F_0}$ is bounded away from $\partial \cR_{G^{F_0}}$.

We can now apply Proposition \ref{prop:hit-any} to 
$(\bn')^{F_0}$ and $(\bn(\ell))^{F_0}$
on the connected graph $G^{F_0}$. We can implement 
the moves given by the strategy in that proposition as a strategy
on $G$, because $d(F_0) = 0$.
Let $\varphi_\ell(\br(\ell))$ denote the probability that 
at time $n(\ell)$ we reach state $\bn(\ell) + \br(\ell)$.
Let us write $c_\ell = \varphi_\ell(\bzero)$ for the probability
that $\bn(\ell)$ was hit exactly. Note that since 
we applied the strategy on $G^{F_0}$, we have $r_e(\ell) = 0$
for all $e \in F_0$. This restriction will be implicit
in our notation. Proposition \ref{prop:hit-any} implies
\eqnspl{e:1st-hit}
{ c_G + \eps
  &\ge p_G(\bn)
  \ge c_\ell p_G(\bn(\ell)) + \sum_{\br(\ell) \not= \bzero} \varphi_\ell(\br(\ell)) \, 
     p_G(\bn(\ell) + \br(\ell)) \\
  &\ge c_\ell (\beta - \eps) + \sum_{0 < |\br(\ell)| < \nu A \sqrt{n(\ell)}} 
     \varphi_\ell(\br(\ell)) \, p_G(\bn(\ell) + \br(\ell)). }
with any $\nu > 0$. The value of $\nu$ will be chosen in what follows.

We now inductively define the strategy that steers 
from $\bn(i+1) + \br(i+1)$ towards $\bn(i)$, for 
$i = \ell-1, \ell-2, \dots, 1$. We assume
$|\br(i+1)| < \nu A \sqrt{n(i+1)}$. Let  
\eqnst
{ H
  = \{ e \in F_0 : n_e(i+1) < \delta_2 A \sqrt{n_{i+1}} \}, }
where $\delta_2 > 0$ will be chosen in a moment. We will first 
reduce the edges in $H$ to their target value $n_e(i)$.
Then we use Proposition \ref{prop:strengthened} and 
Proposition \ref{prop:main-steer} in $G^H$ to reach
a target where the edges $e \in F_0 \setminus H$ do not
have much excess compared to $n_e(i)$, so that these can be 
reduced to $n_e(i)$ as well. Following this, we use
Proposition \ref{prop:hit-any} in $G^{F_0}$ to hit $\bn(i)$.

The first part of the strategy is to reduce the value on each 
edge $e \in H$, whenever that is possible, until it equals $n_e(i)$, 
and in such a way that no edge in $F_0 \setminus H$ is used. 
We stop the first time $t$ when $N_e(t) = n_e(i)$ for all $e \in H$. 
Since $d(F_0) = 0$, such strategy exists. The goal is achieved before time 
$C \delta_2 A \sqrt{n(i+1)}$ with probability 
$\ge 1 - \exp ( - \lambda \sqrt{n(i+1)} )$, if $\delta_2$ is sufficiently small. 
Moreover, the value of every $e \in E \setminus F_0$ is decreased by
no more than $C \delta_2 A \sqrt{n(i+1)}$.
Let $\bn'(i+1)$ denote the configuration reached.

\begin{lemma}
\label{lem:applies}
If $\delta_2$ and $\nu$ are sufficiently small, the restriction of the 
configuration $\bn'(i+1)$ to $G^H$ satisfies the assumption on the starting 
state of Proposition \ref{prop:strengthened} with $A$ replaced by
$\min \{ \frac{1}{2} \kappa_0 A, \delta_2 A \}$.
\end{lemma}

\begin{proof}
The proof is similar to the proof of Lemma \ref{lem:str-applies}. 
Let $\es \subsetneq F \subsetneq E \setminus H$. 
If $d(F \cup H ; G) \ge 1$, we have
\eqnspl{e:ge1}
{ \sum_{e \in F} n'_e(i+1)
  &= \sum_{e \in F \cup H} n'_e(i+1) - \sum_{e \in H} n_e(i) 
  \ge \sum_{e \in F \cup H} n'_e(i+1) - \sum_{e \in H} (n_e(i+1) + r_e(i+1)) \\
  &\ge \sum_{e \in F \cup H} (n_e(i+1) + r_e(i+1)) - (C + |H|) \delta_2 A \sqrt{n(i+1)} \\
  &\ge \sum_{e \in F \cup H} n_e(i+1) - \sqrt{|E|} |\br(i+1)| 
     - (C + |H|) \delta_2 A \sqrt{n(i+1)} \\
  &\ge \frac{n(i+1)}{k} d(F \cup H ; G) + A \sqrt{n(i+1)} - ( \sqrt{|E|} \nu 
     + (C + |H|) \delta_2) A \sqrt{n(i+1)} \\
  &\ge \frac{n'(i+1)}{k} d(F ; G^H) + (1 - C' \nu + C'' \delta_2 ) A \sqrt{n'(i+1)}. }
Hence we will require that $1 - C' \nu - C'' \delta_2 \ge \frac{1}{2}$, say.

When $d(F \cup H ; G) = 0$ and $F$ is not a subset of $F_0$, we have
\eqnspl{e:=0-1}
{ \sum_{e \in F} n'_e(i+1)
  &\ge \sum_{e \in F} (n_e(i+1) + r_e(i+1)) - C \delta_2 A \sqrt{n(i+1)} \\
  &\ge \sum_{e \in F} n_e(i+1) - (\sqrt{|E|} \nu + C \delta_2) A \sqrt{n(i+1)} \\
  &\ge (\kappa_0 - \sqrt{|E|} \nu - C \delta_2 ) A \sqrt{n(i+1)} \\
  &\ge \frac{1}{2} \kappa_0 A \sqrt{n'(i+1)}, }
if $\nu$ and $\delta_2$ are small enough.

Finally, if $\es \subsetneq F \subset F_0 \setminus H$, we have
\eqn{e:=0-2}
{ \sum_{e \in F} n'_e(i+1)
  = \sum_{e \in F} n_e(i+1)
  \ge \sum_{e \in F} \delta_2 A \sqrt{n(i+1)}
  \ge \delta_2 A \sqrt{n'(i+1)}. }
The cases \eqref{e:ge1}, \eqref{e:=0-1} and \eqref{e:=0-2} complete the proof.
\end{proof}

We need one more auxilliary configuration. Let $n''(i) = 2 K_6 n(i)$, 
where $K_6$ is the constant from Proposition \ref{prop:hit-any}, and let 
\eqnst
{ \bn''(i) 
  = K_6 n(i) \by^{*,F_0}(\delta_1; H) 
    + (K_6 - 1) \frac{n(i)}{(n(i))^{H}} (\bn(i))^{H} 
    + \bn(i) + O(1). }
Due to Lemma \ref{lem:x^*-mod}(ii), $\bn''(i)/n''(i) \in \cR_G$ and
$(\bn''(i))^H/(n''(i))^H$ is at least distance $c \delta_1$ away from 
$\partial \cR_{G^H}$. Therefore, we can apply Proposition \ref{prop:main-steer} 
on the graph $G^H$ to steer the process from $(\bn'(i+1))^H$ to a 
$\delta_3$ neighbourhood of $(\bn''(i))^H$, which succeeds with probability at least 
$1 - C_1 \exp ( - \lambda_1 \delta_3 n(i) )$. Moreover, due to Lemma \ref{lem:x^*-mod}(iii),
the configuration $\bn''(i) + \bs$ reached this way satisfies
\eqn{e:not-subset-bound}
{ (2 K_6 n(i))^{-1} \sum_{e \in F} (n''_e(i) + s_e)
  \ge \frac{d(F ; G^{F_0})}{k} + c_7',
     \quad \es \subsetneq F \subsetneq E^{F_0}. }
Also, for $e \in F_0 \setminus H$ we have
\eqnsplst
{ (n''_e(i) + s_e) - n_e(i)
  &\ge K_6 n(i) y^{*,F_0}_e(\delta_1; H) - \sqrt{|E|} |\bs| - \frac{1}{2} A \sqrt{n(i)} \\
  &\ge K_6 n(i) \frac{\delta_1}{|F_0|} 
     - 2 K_6 n(i) \sqrt{|E|} \delta_3 - \frac{1}{2} A \sqrt{n(i)} 
  \ge 0, }
if $\delta_3 < \delta_1 (4 |F_0| \sqrt{|E|})^{-1}$ 
and $n(i)$ is large enough. On the other hand:
\eqnsplst
{ n''_e(i) + s_e
  &\le K_6 n(i) \delta_1 + \sqrt{|E|} |\bs| 
    + K_6 \frac{1}{2} A \sqrt{n(i)} (1 + O(n(i)^{-1/2})) \\
  &\le K_6 n(i) \delta_1 + 2 K_6 n(i) \sqrt{|E|} \delta_3 
  \le 2 K_6 n(i) \delta_1, }
if $n(i)$ is large enough. 

If $\delta_1$ is sufficiently small, we can now employ a strategy 
starting from state $\bn''(i) + \bs$, that reduces the values 
on all $e \in F_0 \setminus H$, whenever that is possible, 
until they all equal $n_e(i)$, but never uses an edge in $H$. 
This only changes the values on $e \in E^{F_0}$ by at most $2 C \delta_1 K_6 n(i)$,
and succeeds with probability at least $1 - \exp ( - \lambda 2 K_6 n(i) )$.
Let $\bn'''(i)$ denote the configuration reached. It follows from
\eqref{e:not-subset-bound} that $(\bn''')^{F_0}$ is bounded away from 
$\partial \cR_{G^{F_0}}$. 

Finally, we can apply Proposition \ref{prop:hit-any} on the graph 
$G^{F_0}$ with starting state $(\bn'''(i))^{F_0}$ 
and target state $(\bn(i))^{F_0}$. Let $\varphi_i(\br(i))$ 
denote the probability that at time $n(i)$ we reach state $\bn(i) + \br(i)$.
Let us write $c_i = \varphi_i(\bzero)$ for the probability
that $\bn(i)$ is hit exactly. This gives the following
inductive bound:
\eqnspl{e:i-th-hit}
{ p_G(\bn(i+1) + \br(i+1))
  &\ge c_i p_G(\bn(i)) + \sum_{\br(i) \not= \bzero} \varphi_i(\br(i)) \, p_G(\bn(i) + \br(i)) \\
  &\ge c_i (\beta - \eps) + \sum_{0 < |\br(i)| < \nu A \sqrt{n_i}} 
      \varphi_i(\br(i)) \, p_G(\bn(i) + \br(i)). } 
Combining \eqref{e:1st-hit} and \eqref{e:i-th-hit}, Proposition \ref{prop:hit-any} 
yields
\eqnsplst
{ c_G + \eps 
  &\ge (\beta - \eps) \left[ c_\ell + (1 - c_\ell) c_{\ell-1} 
    + \dots + (1 - c_\ell) \cdots (1 - c_2) c_1 \right] \\
  &\qquad\qquad - C \ell \exp ( - \lambda A^2 ) - C \exp ( - \lambda \nu A \sqrt{n_1} ). }
Since each $c_j \ge c > 0$, we extract a factor arbitrarily close 
to $\beta - \eps$. Letting $\eps \downarrow 0$ shows that 
$c_G \ge \beta (1 - e^{-c \ell}) - C \ell \exp ( - \lambda A^2)$. 
Choosing $\ell$ of order $A^2$ completes the proof.
\end{proof}

\subsection{Upper bound for $\cB_G^{III}$}
\label{ssec:III}

In the proof of Proposition \ref{prop:near-critical} we are going to
need the following lemma about supermartingales. It is a close variant
of \cite[Propositions 17.19 and 17.20]{LPWbook} and hence we
omit the proof.

\begin{lemma}
\label{lem:super-exit}
Let $Z(t)$ be a non-negative supermartingale with respect to $\cF_t$, 
and $\tau$ a stopping time with respect to $\cF_t$. Suppose that\\
(i) $Z(0) = k \ge 1$; \\
(ii) $|Z(t+1) - Z(t)| \le B$; \\
(iii) there exist constants $\sigma^2 > 0$ and $b > 0$ such that 
almost surely on the event $\{ \tau > t \}$, 
either $\Var ( Z(t+1) \,|\, \cF_t ) \ge \sigma^2$ or 
$\Var(Z(t+1) \,|\, \cF_t ) = 0$ and $\E [ Z(t+1) = Z(t) \,|\, \cF_t ] \le -b$.
Then there exists $u_1 = u_1(B,b,\sigma)$ and $C = C(b,\sigma)$ 
such that if $u \ge u_1$ then 
\eqnst
{ \P [ \tau > u ] 
   \le C \frac{k}{\sqrt{u}}. }
\end{lemma}

\begin{proof}[Proof of Proposition \ref{prop:near-critical}.]
Given $\eps > 0$ choose $A_0(\eps)$ large enough so that the conclusions
of Propositions \ref{prop:far-enough} and \ref{prop:in-enough}
are satisfied for all $A \ge A_0$. Under the optimal strategy, 
we consider the process 
\eqn{e:min-dev}
{ Z(t)
  = \min \{ L^{F,n-t}(\bN(t)) : \, F,\, 0 < d(F) < k \}, }
which is a supermartingale, because the $L^{F,n-t}$ are.
Since the increments of $L^{F,n}$ are bounded, 
condition (ii) of Lemma \ref{lem:super-exit} is satisfied.
We show that $Z(t)$ satisfies the condition (iii) of Lemma \ref{lem:super-exit}
as well. Let $F$ be the set contributing the minimum in \eqref{e:min-dev}.
Since $d(F) > 0$, there exists an edge $e \in F$ such that 
$N_e$ gets updated with probability at least $1/k$. 
On this event we have
\eqnst
{ L^{F,n-t-1}(\bN(t+1)) - L^{F,n-t}(\bN(t))
  = - \langle \bone^e - \bz^F, \bu^F \rangle
  =: - b(e;F) 
  < 0, }
since $d(F) < k$. Therefore, if $\Var ( Z(t+1) \,|\, \cF_t ) = 0$,
we have $\E [ Z(t+1) - Z(t) \,|\, \cF_t ] \le -b(e;F)$.
On the other hand, since there are only finitely many 
possible shifts in the values of the $L^{F,n-t}$, and only
finitely many possible vectors $\bY(t)$ (recall that there
exists a deterministic optimal strategy), if 
$\Var ( Z(t+1) \,|\, \cF_t )$ is non-zero, then 
it is bounded below by some $\sigma^2 = \sigma^2(G) > 0$.

We will choose a small $a > 0$, and subdivide $\cB_G^{III}(n; A)$ into 
the slices: 
\eqnsplst
{ \cB_G^{III}(n; a, k)
  &= \left\{ \bn \in n \cS_G : \, \min \left\{ L^{F,n}(\bn) : 
    \, F,\, 0 < d(F) < k \right\} 
    \in [a k \sqrt{n}, a (k+1) \sqrt{n}) \right\}, \\
  &\qquad\qquad  a > 0,\, -k_\maxim-2 \le k \le k_\maxim+1, }
where $k_\maxim = \lceil A/a \rceil$.
Let $\bn \in \cB_G^{III}(n; a, k)$. The idea of the proof is to 
run the martingale $p_G(\bN(t))$ until $Z(t)$ moves well
into one of the neighbouring slices, and use 
optional stopping to get an inequality 
relating the maximum of $p_G(\bn)$ over 
$\cB_G^{III}(n; a, k)$ to the maxima over
$\cB_G^{III}(n'; a, k-1)$ and $\cB_G^{III}(n'; a, k+1)$,
with $\frac{1}{4} n \le n' < n$. The parameter $a$ will be chosen small 
so that we can apply Lemma \ref{lem:super-exit} to the stopping rule.
We will need to handle $k \ge 1$, $k = 0, -1$ and $k \le -2$ separately.
It will be convenient to introduce the following notation:
\eqnsplst
{ M_n(k)
  &= \max \left\{ p_G(\bn) : \bn \in \cB_G^{III}(n ; a, k) \right\} \\
  \overline{M}_n(k)
  &= \sup_{m \ge n} M_m(k) \\
  \beta(k)
  &= \limsup_{n \to \infty} M_n(k) 
  = \lim_{n \to \infty} \overline{M}_n(k). }

\emph{Case $1 \le k \le k_\maxim$.} We define the stopping time
\eqnsplst
{ \tau_k
  &= \sqrt{a} n \left( \frac{1}{k} - \frac{1}{4 k^2} \right) \, \wedge \, 
    \inf \left\{ t \ge 0 : Z(t) < \left( k - \frac{1}{2} \right) a \sqrt{n-t} \right\} \\
  &\qquad\qquad \wedge \, \inf \left\{ t \ge 0 : 
    Z(t) \ge \left( k + \frac{3}{2} \right) a \sqrt{n-t} \right\}, }
It is straightforward to check that whenever 
$\tau_k < n ( \frac{1}{k} - \frac{1}{4 k^2} )$,
the value of $Z(\tau_k)$ is such that $\bN(\tau_k)$ is either in
the slice $\cB_G^{III}(n-\tau_k ; a, k-1)$ or in the slice
$\cB_G^{III}(n-\tau_k ; a, k+1)$.
An application of Lemma \ref{lem:super-exit} to 
$Z(t) - (k-1) a \sqrt{n}$ yields
\eqn{e:bound-prob}
{ \P \left[ \tau_k \ge \sqrt{a} n \left( \frac{1}{k} - \frac{1}{4 k^2} \right) \right]
  \le C \frac{2 a \sqrt{n}}{a^{1/4} \sqrt{n} \sqrt{ \frac{1}{k} - \frac{1}{4 k^2} }}
  \le C \frac{4 a^{3/4}}{\sqrt{\frac{a}{2 A} \left(4 - \frac{a}{2 A}\right)}}
  = \frac{4 C}{\sqrt{\frac{1}{A} \left(\frac{2}{\sqrt{a}} - \frac{\sqrt{a}}{2 A}\right)}}. }
By optional stopping, we have
\eqnspl{e:pG-mart}
{ p_G(\bn)
  &= \E [ p_G(\bN(\tau_k)) ] \\
  &\le \P [ Z(\tau_k) < k a \sqrt{n - \tau_k} ] \overline{M}_{n/4}(k-1)
    + \P [ Z(\tau_k) \ge (k+1) a \sqrt{n - \tau_k} ] \overline{M}_{n/4} (k+1) \\
  &\qquad + \P [ Z(\tau_k) \in [k a \sqrt{n - \tau_k}, (k+1) a \sqrt{n - \tau_k}) ]
      \overline{M}_{n/4} (k). }
Note that due to our choice of $a$ in \eqref{e:bound-prob} the 
probability in the third term of \eqref{e:pG-mart} is at most 
$C(A) \sqrt{a}$. Maximizing $p_G(\bn)$ over its slice yields
\eqn{e:rel-1}
{ M_n(k) 
  \le c_n(k) \overline{M}_{n/4}(k-1) + d_n(k) \overline{M}_{n/4}(k) + 
      e_n(k) \overline{M}_{n/4}(k+1),
      \quad 1 \le k \le k_\maxim, }
where $d_n(k) \le C(A) \sqrt{a}$. By stopping the supermartingale 
$Z'(t) = Z(t) - (k-1) a \sqrt{n}$ at $\tau_k$ we have
\eqn{e:super-mart-ineq}
{ 2 a \sqrt{n}
  \ge Z'(0)
  \ge \E [ Z'(\tau_k) ; Z'(\tau_k) \ge \frac{5}{2} a \sqrt{n-\tau_k} ] 
  \ge \frac{5}{2} a \sqrt{n} \sqrt{1 - \sqrt{a}} \, e_n(k). }
When $a$ is sufficiently small, the inequalties \eqref{e:super-mart-ineq}
and $d_n(k) \le C(A) \sqrt{a}$ imply that $c_n(k) \ge \frac{1}{6}$.

\emph{Case $k = -1, 0$.} We define
\eqnst
{ \tau_k
  = \frac{3}{4} a n \wedge \inf \left\{ t \ge 0 : 
    Z(t) < \left( k - \frac{1}{2} \right) a \sqrt{n-t} \right\} 
    \wedge \inf \left\{ t \ge 0 : 
    Z(t) \ge \left( k + \frac{3}{2} \right) a \sqrt{n-t} \right\}. }
We now have
\eqn{e:bound-prob-2}
{ \P \left[ \tau_k \ge \frac{3}{4} a n \right]
  \le C \frac{2 a \sqrt{n}}{\sqrt{\frac{3}{4} a n}}
  = \frac{2 \sqrt{a} C}{\sqrt{3/4}}. } 
Analogously to \eqref{e:rel-1} we obtain
\eqn{e:rel-2}
{ M_n(k) 
  \le c_n(k) \overline{M}_{n/4}(k-1) + d_n(k) \overline{M}_{n/4}(k) + 
      e_n(k) \overline{M}_{n/4}(k+1), 
      \quad k = -1, 0. }
By an argument similar to the one for the previous case, for $a$ 
sufficiently small we have $c_n(k) \ge \frac{1}{4}$.

\emph{Case $-k_\maxim-1 \le k \le -2$.} This time we define
\eqnsplst
{ \tau_k
  &= n \sqrt{a} \left( \frac{1}{1-k} - \frac{1}{4 (1-k)^2} \right) \, \wedge \,  
     \inf \left\{ t \ge 0 : Z(t) < \left( k - \frac{1}{2} \right) a \sqrt{n-t} \right\} \\
  &\qquad\qquad \wedge \, \inf \left\{ t \ge 0 : 
     Z(t) \ge \left( k + \frac{3}{2} \right) a \sqrt{n} \right\}. }
Then with the same choice of $a$ as in the case $k \ge 1$ we have
\eqnst
{ \P \left[ \tau_k > n \left( \frac{1}{1-k} - \frac{1}{4 (1-k)^2} \right) \right]
  \le C \frac{4 a^{3/4}}{\sqrt{\frac{a}{2 A} \left(4 - \frac{a}{2 A}\right)}}
  \le C(A) \sqrt{a}. }
This yields the relation
\eqn{e:rel-3}
{ M_n(k) 
  \le c_n(k) \overline{M}_{n/4}(k-1) + d_n(k) \overline{M}_{n/4}(k) + 
      e_n(k) \overline{M}_{n/4}(k+1),
      \quad -k_\maxim-1 \le k \le -2, }
where $c_n(k) \ge \frac{1}{4}$ for sufficiently small $a$.

We select a subsequence of $n$ along which $c_n(k), d_n(k), e_n(k)$ all
converge to some limits $c(k), d(k), e(k)$, as well as all $M_n(k)$ 
converge to $\beta(k)$. Then we get
\eqn{e:beta-ks}
{ \beta(k)
  \le c(k) \beta(k-1) + d(k) \beta(k) + e(k) \beta(k+1), }
Due to Proposition \ref{prop:far-enough} we have $\beta(-k_\maxim-2) \le \eps$ 
and $\beta(k_\maxim+1) \le c_G + \eps$. It is easy to deduce from the relation
\eqref{e:beta-ks} and $c(k) \ge \frac{1}{4} > 0$ that 
if $\beta(k) \ge \beta(k+1)$ then also $\beta(k-1) \ge \beta(k)$.
Hence the maximum in the variable $k$ occurs at the 
right endpoint and $\beta(k) \le c_G + \eps$
for all $-k_\maxim-2 \le k < k_\maxim+1$. This completes the proof of the 
Proposition.
\end{proof}

\section{Further Questions}
\label{sec:open}

\begin{problem}
\label{prob:exp-converge}
It is plausible that the limit $c_G$ is reached at an 
exponential rate everywhere in $\cR_G$.
If one could show that 
$p_G(\bn)$ is maximized in the interior of $\cR_G$,
then this would follow rather easily from \eqref{e:liminf}.
Can one describe the asymptotic behaviour of the optimal
strategy? 
\end{problem}

\begin{problem}
The estimates in Section \ref{sec:critical-proof} strongly
suggest Gaussian behaviour near $\partial \cR_G$. Can one 
make this more precise?
\end{problem}

\begin{problem}
It is plausible that under the optimal strategy, the games starting from
$\bn, \bn' \in n \cR_G$ (and with the same sequence of vertices drawn) 
couple with high probability. This may provide an alternative approach 
to the rather technical arguments of 
Theorem \ref{thm:phase-trans-graph}(ii)
and Proposition \ref{prop:in-enough}. 
\end{problem}

\begin{problem}
We describe a possible definition of an ``order parameter'', in analogy
with statistical physics models. Let $0 \le \alpha \le 1$, and
suppose that the player has to give up proportion $\alpha$ of
her/his moves to an adversary, at which times the  
move is chosen by the adversary. Let $p_{G,\alpha}(\bn)$ denote 
the probability of winning in such a game. Let
\eqnst
{ \theta(\bx)
  = \inf \{ 0 \le \alpha \le 1 : 
    \lim_{n \to \infty} p_{G,\alpha}(n \bx) = 0 \}. }
The methods of Theorem \ref{thm:phase-trans-graph} show that 
$\theta(\bx) > 0$ in $\cR_G$ and $\theta(\bx) = 0$ in $\cI_G$.
Can one analyze $\theta$, or a suitable alternative?
\end{problem}

\bigbreak

\end{document}